\documentclass[a4paper,10pt,leqno, english]{amsart}
\title[MULTIPLICATIVE STRUCTURES  ]
{Multiplicative Structures  and the  Twisted  Baum-Connes  Assembly  map }

\author{No\'e B\'arcenas  }
        \address{Centro de Ciencias Matem\'aticas. UNAM \\ Ap.Postal 61-3 Xangari. Morelia, Michoac\'an M\'EXICO 58089}
         \email{barcenas@matmor.unam.mx}
         \urladdr{http://www.matmor.unam.mx\~barcenas}
         
\author{Paulo Carrillo Rouse}
\address{Institut  de  Math\'ematiques  de  Toulouse. \\118, route de Narbonne F-31062 Toulouse, FRANCE}
         \date{\today}
\email{paulo.carrillo@math.univ-toulouse.fr}
\urladdr{http://www.math.univ-toulouse.fr/~carrillo}
         
\author{Mario Vel\'asquez }
\address{Centro de Ciencias Matem\'aticas. UNAM \\Ap.Postal 61-3 Xangari. Morelia, Michoac\'an M\'EXICO 58089}
\address{Departamento de Matem\'aticas. \\Pontificia Universidad Javeriana\\Cra. 7 No. 43-82 - Edificio Carlos Ort\'iz 5to piso\\ Bogot\'a D.C, Colombia}

 \email{mavelasquezm@gmail.com}
 \urladdr{https://sites.google.com/site/mavelasquezm/}

\usepackage{hyperref}
\usepackage{pdfsync}
\usepackage{calc}
\usepackage{enumerate,amssymb}
\usepackage[arrow,curve,matrix,tips,2cell]{xy}
\usepackage{comment}
  \SelectTips{eu}{10} \UseTips
  \UseAllTwocells

\DeclareMathAlphabet\EuR{U}{eur}{m}{n}
\SetMathAlphabet\EuR{bold}{U}{eur}{b}{n}

\makeindex             

\theoremstyle{plain}
\newtheorem{theorem}{Theorem}[section]
\newtheorem{lemma}[theorem]{Lemma}
\newtheorem{proposition}[theorem]{Proposition}
\newtheorem{corollary}[theorem]{Corollary}

\theoremstyle{definition}
\newtheorem{definition}[theorem]{Definition}
\newtheorem{example}[theorem]{Example}

\newtheorem{remark}[theorem]{Remark}

{\catcode`@=11\global\let\c@equation=\c@theorem}

\newcommand{\comsquare}[8]                   
{\begin{CD}
#1 @>#2>> #3\\
@V{#4}VV @V{#5}VV\\
#6 @>#7>> #8
\end{CD}
}

\newcommand{\xycomsquare}[8]                   
{\xymatrix
{#1 \ar[r]^{#2} \ar[d]^{#4} &
#3 \ar[d]^{#5}  \\
#6\ar[r]^{#7} &
#8
}
}

\newcommand{\calh}{\mathcal{H}}

\newcommand{\calu}{{\mathcal U}}

\newcommand{\IC}{{\mathbb C}}

\newcommand{\IK}{{\mathbb K}}

\newcommand{\IR}{{\mathbb R}}

\newcommand{\IZ}{{\mathbb Z}}

\newcommand{\Hg}{\mathcal{\widehat{H}}}

\newcommand{\Cliff}{\operatorname{{C}liff}}

\newcommand{\Spin}{\operatorname{Spin}}

\newcommand{\SO}{\operatorname{SO}}

\newcommand{\Fred}{\ensuremath{{\mathrm{Fred}}}}
\newcommand{\FredP}{\Fred^{(0)}(\widehat{P})}

\newcommand{\UU}{\mathcal{U}}
\newcommand{\HH}{\mathcal{H}}
\newcommand{\BB}{\mathcal{B}}
\newcommand{\KK}{\mathcal{K}}

\newcommand{\higherlim}[3]{{\setbox1=\hbox{\rm lim}
        \setbox2=\hbox to \wd1{\leftarrowfill} \ht2=0pt \dp2=-1pt
        \mathop{\vtop{\baselineskip=5pt\box1\box2}}
        _{#1}}^{#2}#3}

\newcommand{\version}[1]                       
{\begin{center} last edited on #1\\
last compiled on \today\\http://unal.edu.co/
name of texfile: \jobname
\end{center}
}

\newcounter{commentcounter}

\newcommand{\commentn}[1]
{\stepcounter{commentcounter} {\bf Comment~\arabic{commentcounter}
(by N.): } {\ttfamily #1} }

\begin{document}

  \maketitle

\typeout{----------------------------  linluesau.tex  ----------------------------}


\setcounter{section}{0}

\begin{center}
Abstract
\end{center}

Using a combination of Atiyah-Segal ideas on one side and of Connes and Baum-Connes ideas on the other, we prove that the Twisted geometric K-homology groups of a Lie groupoid  have an external multiplicative structure extending hence the external product structures for proper cases considered by Adem-Ruan in \cite{AdemRuan} or by Tu,Xu and Laurent-Gengoux in \cite{tuxustacks}. These Twisted geometric K-homology groups are the left hand sides of the twisted geometric Baum-Connes assembly maps recently constructed in \cite{CaWangBC} and hence one can transfer the multiplicative structure via the Baum-Connes map to the Twisted K-theory groups whenever this assembly maps are isomorphisms. 


\section{Introduction}
In recent years twisted K-theory and twisted index theory have benefited of a great deal of interest from several groups of mathematicians and theoretical physicists. Besides its relations with string theory and theoretical physics in general, one of the main mathematical motivations was the series of works by Freed, Hopkins and Teleman in which they describe a ring structure on an equivariant twisted K-theory of a group (compact connected Lie group) and in which they give a ring isomorphism with the Verlinde algebra of the group.

For discrete or non compact Lie groups it is not clear how these multiplicative structures should be defined directly or even if they exist at all. In this paper we give a step into trying to understand these issues. Our approach is a mixture of Atiyah-Segal ideas on one side and of Connes and Baum-Connes ideas on the other. Indeed, if the group in question acts properly on a nice space then one can use a homotopy theoretical model for the twisted K-theory groups and use Atiyah-Segal ideas for defining a product in this setting. On the other hand, following Baum-Connes ideas one might expect that the analytically defined twisted equivariant $K$-theory can be approached (or assembled to be precise) by groups defined by using only proper actions (the so called left hand side). The main result of this paper is to define a multiplicative structure on the left hand side of a twisted Baum-Connes assembly map associated to every Lie groupoid, proper or not. We explain this below with more details but before let us mention why we abruptly changed our terminology from groups to groupoids. We have at least two big reasons for this, first, the category of Lie groupoids encodes much more that groups and group actions, many singular situations can be handled using appropriate groupoids; second, our constructions and proofs are largely simplified by the use Connes deformation groupoids techniques (see explanation below).

We pass now to the explicit content of the paper. For proper groupoids one can define the twisted K-theory groups by a generalization of Atiyah-J\"anich Fredholm model for classical topological K-theory. More precisely, if $G$ is a proper Lie groupoid with connected units $M$ and $P$ is a $G$-equivariant $PU(H)-$principal bundle over $M$, the 
{\it{Twisted $G$-equivariant
  K-theory}} groups of $M$ twisted by $P$ can be defined as the  homotopy  groups  of  the  $G$-equivariant  sections
\begin{equation}
K^{-p}_G(M,P) := \pi_p \left( \Gamma(M;\Fred^{(0)}(\widehat{P}))^G, s \right)
\end{equation}
where $\Fred^{(0)}(\widehat{P})\to M$ is a certain bundle constructed from $P$ with fibers an space of Fredholm operators, see definition \ref{definition K-theory of X,P} for more details.
Using this suitable choice of Fredholm bundles we follow Atiyah-Segal for defining a product
\begin{equation}\label{bulletintro}
K^{-p}_G(M,P)\times K^{-q}_G(M,P')\stackrel{\bullet}{\longrightarrow} K^{-(p+q)}_G(M,P\otimes P').
\end{equation}
These twisted K-theory groups for proper groupoids are isomorphic to the K-theory of some $C^*$-algebras  associated to the twisting, theorem 3.14 in \cite{tuxustacks} and section \ref{subsectionTopFred} below. In fact, using Kasparov external product, Tu and Xu construct a product as above in their model for twisted K-theory, they show it gives a bilinear and associative product compatible with the  vector bundle description of twisted K-theory for proper groupoids, theorem 6.1 loc.cit. In this way Tu and Xu generalized the external product defined first by Adem and Ruan in the Orbifold setting in \cite{AdemRuan} (page 552 before definition 7.6).
In proposition \ref{productsequiv} below, we show that modulo the isomorphism between the two twisted K-theory group models our product \ref{bulletintro} above coincides with the one by Tu and Xu, and hence with the one by Adem and Ruan in the Orbifold setting (and for the twistings considered there). In particular the product above is bilinear and associative.

For non necessarily proper groupoids one does not dispose of a Fredholm model for defining the multiplicative structure as above and even if there is a $C^*$-algebraic model for twisted K-theory it is not clear how to define this product directly, the Kasparov external product method mentioned above does not apply since it is not clear how to realize the twisted K-theory groups as appropriate KK-groups. However, following Baum-Connes ideas one might expect that the $K$-theory of the twisted algebra can be approached by K-theory  groups using only proper actions. 

Given a Lie groupoid $G$ (not necessarily proper) together with a class $\alpha\in H^1(G,PU(H))$, the authors in \cite{CaWangBC} formalized and generalized to the twisted case, Connes construction of the geometric K-homology group, denoted by $K_{*}^{geo}(G,\alpha)$, and the construction of the geometric Baum-Connes assembly map from this group to the K-theory group of the $C^*$-algebra $C^*(G,\alpha)$ (reduced or max, the two versions exist). The main theorem in order to prove that this group and the assembly map are well defined is the wrong way functoriality of the pushforward construction in twisted K-theory associated to oriented smooth $G$-maps (theorem 4.2 in \cite{CaWangBC}). 

In this paper we construct a product
\begin{equation}\label{Kgeoprodintro}
K_{*}^{geo}(G,\alpha)\times K_{*}^{geo}(G,\beta)\to K_{*}^{geo}(G,\alpha+\beta)
\end{equation}
that we now explain. First, we are able to describe the groups $K_{*}^{geo}(G,\alpha)$ in terms of the Fredholm picture, that is as the group generated by cycles of the form $(X,x)$ where $X$ is a $G-$proper co-compact manifold (with K-oriented and submersion moment map) and $x\in K_{G}^{-p}(X,P_X)$ (where $P_X$ is the $PU(H)$-bundle over $X$ induced by $P_\alpha$ (a $PU(H)$-bundle representing $\alpha$) and where $K_{G}^{-p}(X,P_X)$ denotes the equivariant twisted $K$-theory group associated to the action groupoid $X\rtimes G$, see definition \ref{definition K-theory of X,P} for more details) and with main relation given by the pushforward maps introduced in \cite{CaWangBC} (see definition \ref{twistkhom} for more precisions) and that we describe here as well in the Fredholm picture. Given two isomorphic $G$-equivariant $PU(H)$-bundles their associated twisted K-theory groups and their associated twisted geometric K-homology groups are isomorphic as well, also the twisted Baum-Connes map mentioned above is compatible with these isomorphisms (theorem 6.4 in \cite{CaWangBC} gives a vast generalization of this fact).

We describe briefly the product before stating the main theorem.
Let $P$ and $Q$ two twistings on $G$. Let $(X,x)$ with $x\in K_{G}^{-p}(X,P_X)$ and $(Y,y)$ with $y\in K_{G}^{-q}(Y,Q_Y)$, the product looks like follows 
\begin{equation}
(X,x)\cdot (Y,y):=(X\times_{G_0}Y,\pi_X^*x\bullet \pi_Y^*y)\in K_{G}^{-p-q}(X\times_{G_0}Y,P_{X\times_{G_0}Y}\otimes Q_{X\times_{G_0}Y})
\end{equation} 
where $\pi_X,\pi_Y$ stand for the respective projections from $X\times_{G_0}Y$ to $X$ and $Y$ and where the pullback is natural operation defined in section \ref{twistsection} below and for which the Fredholm model is very suitable. The main theorem of this paper can be stated as follows:

\begin{theorem}\label{mainthmintro}
For any Lie groupoid $G$ the product on cycles described above gives a well defined bilinear associative product
\begin{equation}\label{Kgeoprodintrothm}
K_{*}^{geo}(G,\alpha)\times K_{*}^{geo}(G,\beta)\to K_{*}^{geo}(G,\alpha+\beta)
\end{equation}
that does not depend on the choices of representatives for $\alpha$ and $\beta$.
\end{theorem}

For proving the theorem above on needs of course pushforward functoriality (that we recall below from \cite{CaWangBC} written in terms of the Fredholm model in section \ref{pushforwardsection}), pullback functoriality (lemma \ref{pullnatural}) and several new technical results as

\begin{enumerate}
\item The compatibility of the product with respect to the pushforward maps, proposition \ref{shriekring} below.
\item The compatibility of the product with respect to the pullback maps, proposition \ref{pullring} below.
\item The compatibilty of the pushforward and the pullback constructions, proposition \ref{shrieknatural} below.
\end{enumerate}

For the properties above the use of the groupoid language becomes very useful. First of all the construction of the pushforward maps can be completely realized in the Fredholm picture by using Connes deformation groupoids, and hence adapting to this model the main results and constructions from \cite{CaWangBC} for the case of proper groupoids, we explain this in section \ref{pushforwardsection}. Second, the proofs become conceptually very simple, for example to prove the first property above amounts to check that the morphisms induced by restriction are compatible with the product. 
So even if one is only interested in the group case (Lie or discrete for instance) the use of deformation groupoids gives a unified way to construct the pushforward maps, to prove their functoriality and to prove their compatibility with the product.

But what can we say about the multiplicative structures in Twisted K-theory directly. By the results above one could expect to transpose the multiplicative structure via the assembly map
\begin{equation}
\xymatrix{
K_{*}^{geo}(G,\alpha)\ar[rr]^-{\mu_{\alpha}}&&K_*(C^*_r(G,\alpha))
}
\end{equation}
constructed in \cite{CaWangBC}. This is of course the case when these twisted Baum-Connes map are isomorphisms. Hence we have a unique bilinear associative structure on the Twisted $K$-theory groups (below $K^{*}(G,\alpha):=K_{-*}(C^*_r(G,\alpha))$)
\begin{equation}\label{prodBCintro}
K^{*}(G,\alpha)\times K^{*}(G,\beta)\to K^*(G,\alpha+\beta)
\end{equation}
compatible with the structure of \ref{mainthmintro} via the assembly maps whenever all the assembly maps $\mu_\alpha$ are isomorphisms (corollary \ref{CorBCextprod}). Now, by corollary 7.2 in \cite{CaWangBC} an assembly map $\mu_\alpha$ is an isomorphism if and only if the assembly map for the associated extension groupoid is. In particular if the geometric assembly map coincides with the analytic assembly map, one might expect that for groupoids (or groups) for which the analytic assembly is known to be an isomorphism for the respective extensions we do have that the multiplicative structure above transfer to the $K$-theory counterpart. This is for example the case for (Hausdorff) Lie groupoids satisfying the Haagerup property (\cite{Tu2} theorem 9.3, see also \cite{Tu3} theorem 6.1). We have then to study the comparison between the geometric and analytic assemblies which are expected to coincide whenever they do in the untwisted case (for discrete groups and for Lie groups).
Another interesting question would be if it is possible to construct directly these multiplicative structures on the Total twisted K-theory groups such that the assembly map is a ring/module isomorphism. These questions will be discussed elsewhere.

The external products discussed above suggest a ring/module structure reflecting in twistings the group structure of $H^1(G; PU(H))$. We discuss this in the last section in which we consider the so called Total Twisted $K-$theory ($K-$homology resp.) groups. These groups and their associated multiplicative structures appeared first in \cite{AdemRuan} in the setting of Orbifolds (definition 8.1 loc.cit).

Finally, the product (\ref{Kgeoprodintrothm}) above is the first step into trying to understand, in the non proper case, internal stringy products in groups of the form $K^{geo}_{G,*}(N,\alpha)$
 (or more generally on the K-theory counterpart) where $N$ is a crossed module (for instance $G$ itself on which $G$ acts by conjugation, in the case of a group) over $G$ and $\alpha$ a twisting with good multiplicative properties (transgressive). Indeed, in all the versions of stringy products (or internal products) one passes necessarily by a product as above before making use of the crossed module structure and of the multiplicativity of the twisting, \cite{FHT1}, \cite{AdemRuanZhang}, \cite{CWfusion}, \cite{TXring} for mention some of them.

\subsection*{Acknowledgements}
The first  author  thanks  the  support  of  DGAPA  research  grant IA100315. The third Author  thanks  the support  of  a  UNAM  Postdoctoral  Fellowship. 

The  first  and  third  author thank  Unversit\'e  Toulouse  III Paul Sabatier, as  well  as  the  Laboratoire International  Solomon  Lefschetz (LAISLA) and the ANR project KIND for  support during  a visit to  Toulouse. The second author thanks the Max Planck Institute for Mathematics in Bonn for giving him excellent conditions for work out part of this project.

\section{Preliminaries on groupoids}

In this section, we review the notion of twistings on Lie groupoids and 
discuss some examples which appear in this paper.
Let us recall what a groupoid is:

\begin{definition}
A $\it{groupoid}$ consists of the following data:
two sets $G$ and $M$, and maps
\begin{itemize}
\item[(1)]  $s,r:G \rightarrow M$ 
called the source map and target map respectively,
\item[(2)]  $m:G^{(2)}\rightarrow G$ called the product map 
(where $G^{(2)}=\{ (\gamma,\eta)\in G \times G : s(\gamma)=r(\eta)\}$),
\end{itemize}
together with  two additional  maps, $u:M \rightarrow G$ (the unit map) and 
$i:G \rightarrow G$ (the inverse map),
such that, if we denote $m(\gamma,\eta)=\gamma \cdot \eta$, $u(x)=x$ and 
$i(\gamma)=\gamma^{-1}$, we have 
\begin{itemize}
\item[(i)] $r(\gamma \cdot \eta) =r(\gamma)$ and $s(\gamma \cdot \eta) =s(\eta)$.
\item[(ii)] $\gamma \cdot (\eta \cdot \delta)=(\gamma \cdot \eta )\cdot \delta$, 
$\forall \gamma,\eta,\delta \in G$ whenever this makes sense.
\item[(iii)] $\gamma \cdot u(x) = \gamma$ and $u(x)\cdot \eta =\eta$, $\forall
  \gamma,\eta \in G$ with $s(\gamma)=x$ and $r(\eta)=x$.
\item[(iv)] $\gamma \cdot \gamma^{-1} =u(r(\gamma))$ and 
$\gamma^{-1} \cdot \gamma =u(s(\gamma))$, $\forall \gamma \in G$.
\end{itemize}
For simplicity, we denote a groupoid by $G \rightrightarrows M $.

\end{definition}

In  this paper we will only deal with Lie groupoids, that is, 
a groupoid in which $G$ and $M$ are smooth manifolds, and $s,r,m,u$ are smooth maps (with s and r submersions). 

\subsection{The Hilsum-Skandalis category}\label{HScat}
 
Lie groupoids form a category with  strict  morphisms of groupoids. It is now a well-established fact  in Lie groupoid's theory that the right category to consider is the one in which Morita equivalences correspond precisely to isomorphisms.  We review some basic definitions and properties of generalized morphisms between Lie groupoids, see \cite{tuxustacks} section 2.1, or 
\cite{HS,Mr,MM} for more detailed discussions.

\begin{definition}[Generalized homomorphisms]\label{HSmorphism}   
Let $G \rightrightarrows M$ and  
$H \rightrightarrows M'$ be two Lie groupoids.  A generalized groupoid morphism, also called a Hilsum-Skandalis morphism, from $H$ to $G$ is given by the isomorphism class of a principal $G$-bundle over $H$, that 
is, a right  principal $G$-bundle over $M'$
which is also a left $H$-bundle over $M$ such that the   the right $G$-action and the left 
$H$-action commute,  formally denoted by
\[
f:  \xymatrix{H \ar@{-->}[r] &  G}
\]
or by  
\[
\xymatrix{
H \ar@<.5ex>[d]\ar@<-.5ex>[d]&P_f \ar@{->>}[ld] \ar[rd]&G \ar@<.5ex>[d]\ar@<-.5ex>[d]\\
M'&&M.
}
\]
if   we want to emphasize  the bi-bundle $P_f$ involved. 
\end{definition}

As the name suggests,  generalized morphism  generalizes the notion of strict morphisms and can be composed. Indeed, if $P$ and $P'$ give generalized morphisms from $H$ to $G$ and from $G$ to $L$ respectively, then 
$$P\times_{G}P':=P\times_{M}P'/(p,p')\sim (p\cdot \gamma, \gamma^{-1}\cdot p')$$
gives a generalized morphism from $H$ to $L$.  Consider the category $Grpd_{HS}$ with objects Lie groupoids and morphisms given by  generalized morphisms. There is a functor
\begin{equation}\label{grpdhs}
Grpd \longrightarrow Grpd_{HS}
\end{equation}
where $Grpd$ is the strict category of groupoids. 

\begin{definition}[Morita equivalent groupoids]
Two groupoids are called Morita equivalent if they are isomorphic in $Grpd_{HS}$. 
\end{definition}

\vspace{2mm}

We list here a few examples of Morita equivalence groupoids which will be used in this paper. 

\begin{example}[Pullback groupoid]
Let $G\rightrightarrows M$ be a Lie groupoid and let $\phi:M\to M$ be a map such that $t\circ pr_2:M\times_{M}G\to M$ is a submersion (for instance if $\phi$ is a submersion), then the pullback groupoid $\phi^*G:=M\times_{M}G\times_{M} M\rightrightarrows M$ is Morita equivalent to $G$, the strict morphism $\phi^*G\to G$ being a generalized isomorphism. For more details on this example the reader can see \cite{MM} examples 5.10(4).
\end{example}

\begin{example}[Discrete groups]
Let $\Gamma$ be a discret group. Let $M$ be a manifold together with a generalized morphis�m
$$M--->\Gamma$$ (in this case this is equivalent a continuous map $M\to B\Gamma$) given by a $\Gamma$-principal bundle $\widetilde{M}\to M$ over $M$ ({\it i.e.}, a $\Gamma$-covering). 
Consider the (Connes-Moscovici) groupoid 
$$\widetilde{M}\times_\Gamma\widetilde{M}\rightrightarrows M$$
where $\widetilde{M}\times_\Gamma\widetilde{M}:=\widetilde{M}\times \widetilde{M}/\triangle{\Gamma}$ and with structural maps $s(\tilde{x},\tilde{y})=y$, $t(\tilde{x},\tilde{y})=x$ and product 
$$(\tilde{x},\tilde{y})\cdot (\tilde{y},\tilde{z}):=(\tilde{x},\tilde{z}).$$
The groupoids $\widetilde{M}\times_\Gamma\widetilde{M}\rightrightarrows M$ and $\Gamma\rightrightarrows \{e\}$ are Morita equivalent. 
\end{example}
 
 \subsection{Twistings on  Lie groupoids}
 
 In this paper,  we are only going to consider $PU(H)$-twistings on Lie groupoids 
where $H$ is an infinite dimensional, complex and separable
Hilbert space, and $PU(H)$ is the projective unitary group $PU(H)$  with the topology induced by the
norm topology on the unitary group  $U(H)$. 

\begin{definition}\label{twistedgroupoid}
A  twisting $\alpha$  on a   Lie  groupoid $G \rightrightarrows M$  is given by  a generalized morphism 
\[ \xymatrix{
\alpha: G \ar@{-->}[r]  & PU(H).}
\]
Here $PU(H)$ is viewed  as a Lie groupoid with the unit space $\{e\}$. 
\end{definition}
 
 So a twisting on a Lie groupoid $G$ is  given by a locally trivial  right  principal $PU(H)$-bundle $P_{\alpha}$ over $G$.

 \begin{remark}
 The definition of generalized morphisms given in the last subsection was for two Lie groupoids. The group $PU(H)$ it is not a  finite dimensional Lie group but it makes perfectly sense to speak of generalized morphisms from Lie groupoids to this infinite dimensional   groupoid following exactly the same definition.
 \end{remark}

\begin{example} \label{example} For a list of  various twistings on some   standard groupoids see example 1.8 in \cite{CaWangAdv}. Here we will only  a few  basic examples.
  
 \begin{enumerate}
\item (Twisting on manifolds)  Let $X$ be a $C^\infty$-manifold. We can consider the  Lie groupoid 
 $X\rightrightarrows X$  where every morphism is the identity over $X$.  A twisting on $X$ is
given by a locally trivial principal $PU(H)$-bundle over $X$.
In particular, the restriction of a twisting $\alpha$ on a  Lie groupoid $G \rightrightarrows M$
to its unit $M$ defines a twisting  $\alpha_0$ on the manifold $M$.

\item\label{obundle} (Orientation twisting) Let $X$ be a  manifold with an oriented real vector bundle $E$. The  bundle $E \to X$ defines
a natural generalized morphism 
\[
\xymatrix{
X\ar@{-->}[r] & SO(n).}
\]
Note that the fundamental unitary  representation of   $Spin^c(n)$ gives rise to a commutative
diagram of Lie group homomorphisms
\[
\xymatrix{
Spin^c(n) \ar[d]   \ar[r] & U(\mathbb{C}^{2^n}) \ar[d] \\
SO(n) \ar[r] & PU(\mathbb{C}^{2^n}).}
\]
With a choice of inclusion $\mathbb{C}^{2^n}$ into a Hilbert space $H$, we have a canonical
twisting, called the orientation twisting, denoted by
\[\xymatrix{
\beta_{E}:  X\ar@{-->}[r] & PU(H).}
\]

\item (Pull-back twisting) Given a twisting $\alpha$ on $ G$ and  for any generalized 
homomorphism $\phi: H \to G$, there is a pull-back twisting 
\[\xymatrix{
\phi^*\alpha:  H  \ar@{-->}[r]  & PU(H)}
\]
defined by the composition of $\phi$ and $\alpha$.  In particular, 
for a continuous map $\phi: X\to Y$, a twisting $\alpha$ on $Y$ gives a pull-back twisting 
$\phi^*\alpha$ on $X$. The principal $PU(H)$-bundle over $X$ defines by $\phi^*\alpha$ is
the pull-back of the  principal $PU(H)$-bundle on $Y$ associated to $\alpha$.

\item (Twisting on fiber product groupoid)  Let $N\stackrel{p}{\rightarrow} M$ be a submersion. We consider the fiber product $N\times_M N:=\{ (n,n')\in N\times N :p(n)=p(n') \}$,which is a manifold because $p$ is a submersion. We can then take the groupoid 
$$N\times_M N\rightrightarrows N$$ which is  a subgroupoid of the pair groupoid 
$N\times N \rightrightarrows N$.  Note that this groupoid is in fact Morita equivalent to the groupoid $ M \rightrightarrows M$.  A twisting on  
 $N\times_M N\rightrightarrows N$ is  given by
a pull-back twisting from a   twisting on  $M$.  

\item (Twisting on a Lie group)
By definition a twisting on a Lie group $G$ is a projective representation
$$G\stackrel{\alpha}{\longrightarrow} PU(H).$$

\end{enumerate}
\end{example}
 
\subsection{Deformation groupoids}\label{defgrpds}

One of our main tools will be the use of deformation groupoids. In this section, we review the notion of Connes' deformation groupoids from the deformation to the normal cone point of view.

{\bf Deformation to the normal cone}\label{DCN}

Let $M$ be a $C^\infty$-manifold and $X\subset M$ be a $C^\infty$-submanifold. We denote
by $\mathcal{N}_{X}^{M}$ the normal bundle to $X$ in $M$.
We define the following set
\begin{align}
\mathcal{D}_{X}^{M}:= \left( \mathcal{N}_{X}^{M} \times {0} \right) \bigsqcup   \left(M \times \mathbb{R}^* \right). 
\end{align} 
The purpose of this section is to recall how to define a $C^\infty$-structure in $\mathcal{D}_{X}^{M}$. This is more or less classical, for example
it was extensively used in \cite{HS}.

Let us first consider the case where $M=\mathbb{R}^p\times \mathbb{R}^q$ 
and $X=\mathbb{R}^p \times \{ 0\}$ ( here we
identify  $X$ canonically with $ \mathbb{R}^p$). We denote by
$q=n-p$ and by $\mathcal{D}_{p}^{n}$ for $\mathcal{D}_{\mathbb{R}^p}^{\mathbb{R}^n}$ as above. In this case
we   have that $\mathcal{D}_{p}^{n}=\mathbb{R}^p \times \mathbb{R}^q \times \mathbb{R}$ (as a
set). Consider the 
bijection  $\psi: \mathbb{R}^p \times \mathbb{R}^q \times \mathbb{R} \rightarrow
\mathcal{D}_{p}^{n}$ given by 
\begin{equation}\label{psi}
\psi(x,\xi ,t) = \left\{ 
\begin{array}{cc}
(x,\xi ,0) &\mbox{ if } t=0 \\
(x,t\xi ,t) &\mbox{ if } t\neq0
\end{array}\right.
\end{equation}
whose  inverse is given explicitly by 
$$
\psi^{-1}(x,\xi ,t) = \left\{ 
\begin{array}{cc}
(x,\xi ,0) &\mbox{ if } t=0 \\
(x,\frac{1}{t}\xi ,t) &\mbox{ if } t\neq0
\end{array}\right.
$$
We can consider the $C^\infty$-structure on $\mathcal{D}_{p}^{n}$
induced by this bijection.

We pass now to the general case. A local chart 
$(\mathcal{U},\phi)$ of $M$ at $x$  is said to be a $X$-slice   if 
\begin{itemize}
\item[1)]  $\mathcal{U}$  is an open neighbourhood of $x$ in $M$ and  $\phi : \mathcal{U}  \rightarrow U \subset \mathbb{R}^p\times \mathbb{R}^q$ is a diffeomorphsim such that $\phi(x) =(0, 0)$. 
\item[2)]  Setting $V =U \cap (\mathbb{R}^p \times \{ 0\})$, then
$\phi^{-1}(V) =   \mathcal{U} \cap X$ , denoted by $\mathcal{V}$.
\end{itemize}
With these notations understood, we have $\mathcal{D}_{V}^{U}\subset \mathcal{D}_{p}^{n}$ as an
open subset.   For $x\in \mathcal{V}$ we have $\phi (x)\in \mathbb{R}^p
\times \{0\}$. If we write 
$\phi(x)=(\phi_1(x),0)$, then 
$$ \phi_1 :\mathcal{V} \rightarrow V \subset \mathbb{R}^p$$ 
is a diffeomorphism.  Define a function 
\begin{equation}\label{phi}
\tilde{\phi}:\mathcal{D}_{\mathcal{V}}^{\mathcal{U}} \rightarrow \mathcal{D}_{V}^{U} 
\end{equation}
by setting 
$\tilde{\phi}(v,\xi ,0)= (\phi_1 (v),d_N\phi_v (\xi ),0)$ and 
$\tilde{\phi}(u,t)= (\phi (u),t)$ 
for $t\neq 0$. Here 
$d_N\phi_v: N_v \rightarrow \mathbb{R}^q$ is the normal component of the
 derivative $d\phi_v$ for $v\in \mathcal{V}$. It is clear that $\tilde{\phi}$ is
 also a  bijection. In particular,  it induces a $C^{\infty}$ structure on $\mathcal{D}_{\mathcal{V}}^{\mathcal{U}}$. 
Now, let us consider an atlas 
$ \{ (\mathcal{U}_{\alpha},\phi_{\alpha}) \}_{\alpha \in \Delta}$ of $M$
 consisting of $X-$slices. Then the collection $ \{ (\mathcal{D}_{\mathcal{V}_{\alpha}}^{\mathcal{U}_{\alpha}},\tilde{\phi}_{\alpha})
  \} _{\alpha \in \Delta }$ is a $C^\infty$-atlas of
  $\mathcal{D}_{X}^{M}$ (Proposition 3.1 in \cite{Ca4}).

\begin{definition}[Deformation to the normal cone]
Let $X\subset M$ be as above. The set
$\mathcal{D}_{X}^{M}$ equipped with the  $C^{\infty}$ structure
induced by the atlas of  $X$-slices is called
 the deformation to the  normal cone associated  to   the embedding
$X\subset M$. 
\end{definition}


One important feature about the deformation to the normal cone is the functoriality. More explicitly,  let
 $f:(M,X)\rightarrow (M',X')$
be a   $C^\infty$-map   
$f:M\rightarrow M'$  with $f(X)\subset X'$. Define 
$ \mathcal{D}(f): \mathcal{D}_{X}^{M} \rightarrow \mathcal{D}_{X'}^{M'} $ by the following formulas: \begin{enumerate}
\item[1)] $\mathcal{D}(f) (m ,t)= (f(m),t)$ for $t\neq 0$, 

\item[2)]  $\mathcal{D}(f) (x,\xi ,0)= (f(x),d_Nf_x (\xi),0)$,
where $d_Nf_x$ is by definition the map
\[  (\mathcal{N}_{X}^{M})_x 
\stackrel{d_Nf_x}{\longrightarrow}  (\mathcal{N}_{X'}^{M'})_{f(x)} \]
induced by $ T_xM 
\stackrel{df_x}{\longrightarrow}  T_{f(x)}M'$.
\end{enumerate}
 Then $\mathcal{D}(f):\mathcal{D}_{X}^{M} \rightarrow \mathcal{D}_{X'}^{M'}$ is a $C^\infty$-map (Proposition 3.4 in \cite{Ca4}). In the language of categories, the deformation to the normal cone  construction defines a functor
\begin{equation}\label{fundnc}
\mathcal{D}: \mathcal{C}_2^{\infty}\longrightarrow \mathcal{C}^{\infty} ,
\end{equation}
where $\mathcal{C}^{\infty}$ is the category of $C^\infty$-manifolds and $\mathcal{C}^{\infty}_2$ is the category of pairs of $C^\infty$-manifolds.

Given  an immersion of Lie groupoids $G_1\stackrel{\varphi}{\rightarrow}G_2$, let 
$G^N_1=\mathcal{N}_{G_1}^{G_2}$ be the total space of the normal bundle to $\varphi$, and $(G_{1}^{(0)})^N$ be the total space of the normal bundle to $\varphi_0: G_{1}^{(0)} \to G_{2}^{(0)}$.  Consider $G^N_1\rightrightarrows (G_{1}^{(0)})^N$ with the following structure maps: The source map is the derivation in the normal direction 
$d_Ns:G^N_1\rightarrow (G_{1}^{(0)})^N$ of the source map (seen as a pair of maps) $s:(G_2,G_1)\rightarrow (G_{2}^{(0)},G_{1}^{(0)})$ and similarly for the target map.

The groupoid $G^N_1$ may fail to inherit a Lie groupoid structure (see counterexample just before section IV in \cite{HS}). A sufficient condition is when $(G_{1}^{(0)})^N$ is a $G^N_1$-vector bundle over $G_{1}^{(0)}$. This is the case when $G_1^x\to G_{2}^{\varphi(x)}$ is \'etale for every $x\in G_{1}^{(0)}$ (in particular if the groupoids are \'etale) or when one considers a manifold with two foliations $F_1\subset F_2$ and the induced immersion  (again 3.1, 3.19 in \cite{HS}).

 The deformation to the normal bundle construction allows us to consider a $C^{\infty}$ structure on 
$$
G_{\varphi}:=\left( G^N_1\times \{0\} \right) \bigsqcup  \left( G_2\times \mathbb{R}^*\right),
$$
such that $G^N_1\times \{0\}$ is a closed saturated submanifold and so $G_2\times \mathbb{R}^*$ is an open submanifold.
The following results are  an immediate consequence of the functoriality of the deformation to the normal cone construction.

\begin{proposition}[Hilsum-Skandalis, 3.1, 3.19 \cite{HS}]\label{HSimmer}
Consider an immersion $G_1\stackrel{\varphi}{\rightarrow}G_2$ as above for which $(G_1)^N$ inherits a Lie groupoid structure.
Let $G_{\varphi_0}:= \big( (G_{1}^{(0)})^N\times \{0\} \big) \bigsqcup \big(  G_{2}^{(0)}\times \mathbb{R}^* \big)$ be the deformation to the normal cone of  the  pair $(G_{2}^{(0)},G_{1}^{(0)})$. The groupoid
\begin{equation}
G_{\varphi}\rightrightarrows G_{\varphi_0}
\end{equation}
with structure maps compatible  with the ones of the groupoids $G_2\rightrightarrows G_{2}^{(0)}$ and $G_1^N\rightrightarrows (G_{1}^{(0)})^N$, is a Lie groupoid with $C^{\infty}$-structures coming from  the deformation to the normal cone.
\end{proposition}

One of the interest of these kind of groupoids is to be able to define family indices. First we recall the following elementary  result.

\begin{proposition}\label{deftwistimm}
Given an immersion of Lie groupoids $G_1\stackrel{\varphi}{\rightarrow}G_2$  as above and a twisting $\alpha$ on $G_2$.  There is a canonical twisting $\alpha_\varphi$ on the Lie groupoid 
$G_{\varphi} \rightrightarrows G_{\varphi_0}$,  extending the pull-back  twisting on $G_2\times \mathbb{R}^*$ from $\alpha$.
\end{proposition}

\begin{proof}
The proof is a simple application of the functoriality of the deformation to the normal cone construction. Indeed, the twisting $\alpha$ on $G_2$ induces by pullback (or composition of cocycles) a twisting $\alpha\circ \varphi$ on $G_1$. The twisting $\alpha$ on $G_2$ is given by a $PU(H)$-principal bundle $P_{\alpha}$ with a compatible left action of $G_2$, and by definition the twisting $\alpha\circ \varphi$ on $G_1$ is given by the pullback of $P_{\alpha}$ by 
$\varphi_0: G_{1}^{(0)} \to G_{2}^{(0)}$. In particular, $P_{\alpha\circ \varphi}=G_{1}^{(0)}\times_{G_{2}^{(0)}}P_{\alpha}$ 
Hence the action map 
$G_2\times_{G_{2}^{(0)}}P_{\alpha}\to P_{\alpha}$ can be considered as an application in the category of pairs:

$$(G_2\times_{G_{2}^{(0)}}P_{\alpha},G_1\times_{G_{1}^{(0)}}P_{\alpha\circ\varphi})\longrightarrow
(G_{2}^{(0)}\times_{G_{2}^{(0)}}P_{\alpha},G_{1}^{(0)}\times_{G_{1}^{(0)}}P_{\alpha\circ \varphi}).$$

We can then apply the deformation to the normal cone functor to obtain the desire $PU(H)$-principal bundle with a compatible $G_{\varphi}$-action, which gives the desired twisting.
\end{proof}


\section{Twisted equivariant K-theory}
 The  crucial  diference  to  \cite{barcenasespinozajoachimuribe} is  the  use  of  graded  Fredholm  bundles, which  are  needed  for  the   definition  of  the multiplicative  structure.  

Let $\HH$ be a separable Hilbert space and 
$$\UU(\HH):= \{ U : \HH \to \HH \mid U\circ U^*= U^*\circ U = \mbox{Id} \}$$ the group
of unitary operators acting on $\HH$. As is noted in \cite{atiyahsegal} there are some issues when consider the \emph{norm} topology, then we use the \emph{compact-open} topology (for an account of the compact-open topology see \cite[Appendix 1]{atiyahsegal}).  Let ${\rm End}(\HH)$ denote the space of endomorphisms
of the Hilbert space and endow ${\rm End}(\HH)_{c.o.}$ with the compact open topology. Consider the inclusion
\begin{align*}
\UU(\HH) &\to {\rm End}(\HH)_{c.o.} \times {\rm End}(\HH)_{c.o.}\\
U &\mapsto (U,U^{-1})
\end{align*}
and induce on $\UU(\HH)$ the subspace topology. Denote the space of unitary operators with this induced topology
by $\UU(\HH)_{c.o.}$ and note that this is different from the usual compact open topology on $\UU(\HH)$. Unfortunately the group $\UU(\HH)_{c.o}$ fails to be a topological group, the composition is continuous only on compact subspaces. Let  $\UU(\HH)_{c.g}$  be  the  compactly  generated  topology  associated  to the  compact  open  topology,  and  topologize the  group $P\UU(\HH)$  from the exact  sequence
$$1\to S^1\to  \UU(\HH)_{c.g.}\to P\UU(\HH)\to  1 .$$


\begin{definition}
Let $\HH$ be a separable Hilbert  space. The  space  $\Fred'(\HH)$
consist of pairs $(A,B)$ of bounded operators on $\HH$ such that
$AB -1$ and $BA -1$ are compact operators. Endow $\Fred'(\HH)$
with the topology induced by the embedding \begin{eqnarray*}
\Fred'(\HH) & \to & {\mathsf{B}}(\HH) \times  {\mathsf{B}}(\HH)
\times {\mathsf{K}}(\HH)
\times {\mathsf{K}}(\HH) \\
(A,B) & \mapsto & (A,B,AB-1, BA-1)
\end{eqnarray*}
where ${\mathsf{B}}(\HH)$ denotes the bounded operators on $\HH$ with
the compact open topology and ${\mathsf{K}}(\HH)$ denotes the compact
operators with the norm topology.

\end{definition}

We denote by $\Hg=\HH\oplus \HH$ a $\IZ_2$-graded, infinite  dimensional  Hilbert space.

\begin{definition}

Let $U(\Hg)_{c.g.}$ be  the  group  of  even, unitary  operators  on the  Hilbert  space $\Hg$ which  are  of  the  form 
$$\begin{pmatrix}  u_1& 0  \\ 0 & u_2      \end{pmatrix},$$

where $u_i$ denotes  a  unitary  operator  in the  compactly generated  topology defined as before. 

We  denote  by  
$P\UU(\Hg)$ the  group $U(\Hg)_{c.g.}/ S^1$ and recall the  central  extension 
$$1\to  S^1\to \UU(\Hg)\to P\UU(\Hg)\to 1$$ 

\end{definition}

\begin{definition}
The space $\Fred''(\Hg)$ is  the  space  of pairs $(\widehat{A},\widehat{B})$ of self-adjoint, bounded operators of degree 1 defined on $\Hg$ such that $\widehat{A}\widehat{B}-I$ and $\widehat{B}\widehat{A}-I$ are compact. 
\end{definition}

Given a $\IZ/2$-graded Hilbert space  $\Hg$, 
the space $\Fred''(\Hg)$ is homeomorphic to $\Fred'(H)$.

\begin{definition}
We denote by $\Fred^{(0)}(\Hg)$ the space of self-adjoint degree 1 Fredholm operators ${A}$ in $\Hg$ such that ${A}^2$ differs from the identity by a compact operator, with the topology coming from the embedding ${A}\mapsto ({A},{A}^2-I)$ in $\BB(\HH)\times\KK(\HH)$.

\end{definition}

The  following  result  was  proved  in \cite{atiyahsegal}, Proposition 3.1 : 

\begin{proposition}
The space $\Fred^{(0)}(\Hg)$ is a deformation retract of $\Fred''(\Hg)$.
\end{proposition}

  
In particular, the above discussion implies that $\Fred^{(0)}(\Hg)$ is a representing space for $K$-theory. The group $\UU(\Hg)_{c.g.}$ of degree 0 unitary operators on $\Hg$ with the compactly generated topology  acts continuously by conjugation on $\Fred^{(0)}(\Hg)$, therefore the group $P\calu(\Hg)$ acts continuously on $\Fred^{(0)}(\Hg)$ by conjugation. In \cite{barcenasespinozajoachimuribe} twisted $K$-theory for proper actions of discrete groups was defined using the representing space $\Fred '(\HH)$, but in order to have multiplicative structure we proceed using $\Fred^{(0)}(\Hg)$.

Let us choose the operator 
\begin{equation*}
  \widehat{I}=\begin{pmatrix}
  0&I\\I&0
  \end{pmatrix}.
  \end{equation*}

as the base point in $\Fred^{(0)}(\Hg)$.

Choosing  the identity  as  a  base point on the space  $\Fred^{'}(\mathcal{H})$,  gives a  diagram  of  pointed  maps

$$\xymatrix{\Fred^{(0)}(\Hg)\ar[r]^{i} &  \Fred^{''}(\Hg)\ar[d]^{r}\ar[r]^f & \Fred^{'}(\HH) \\ & \Fred^{(0)}(\Hg) &  } , $$
where  $i$  denotes  the inclusion, $r$   is  a  strong  deformation  retract  and $f$  is a  homeomorphism. Moreover,  the  maps  are  compatible  with the conjugation  actions  of  the  groups $\UU(\Hg)_{c.g.}$,  $\UU(\HH)_{c.g.}$ and  the  map $ \UU(\Hg)_{c.g.} \to \UU(\HH)_{c.g.}$.

Let $X$ be a proper $G$-space and  let $P \to X$  be a projective unitary 
$G$-equivariant bundle over $X$. Denote by $\widehat{P}$  the  projective  unitary  bundle  obtained  by  performing  the tensor  product  with the  trivial bundle $\mathbb{P}(\Hg)$, $\widehat{P}= P\otimes\mathbb{P}(\Hg)$.

 The space of Fredholm
operators  is endowed with a continuous right action
of the group $P\calu(\widehat{\calh})$ by conjugation, therefore we can take
the associated bundle over $X$
$$\Fred^{(0)}(\widehat{P}) := \widehat{P} \times_{P\calu(\widehat{\calh})} \Fred^ {(0)}(\widehat{\calh}),$$
  and with the induced $G$ action given by
 $$g \cdot [(\lambda, A))] := [(g \lambda,A)]$$
for $g$ in $G$, $\lambda$ in $\widehat{P}$ and $A$ in $\Fred^{(0)}(\widehat{\calh})$.

Denote by $$\Gamma(X; \Fred^{(0)}(\widehat{P}))$$ the space of sections of the
bundle $\Fred^{(0)}(\widehat{P}) \to X$ and choose as base point in this space the
section which chooses the base point $\widehat{I}$ on the fibers. This section
exists because the ${P}\calu(\widehat{\calh})$ action on ${\widehat{I}}$
is trivial, and therefore $$X \cong \widehat{P}/P\calu(\widehat{\calh}) \cong \widehat{P}
\times_{P\calu(\widehat{\calh})} \{{\widehat{I}} \} \subset \Fred^{(0)}(\widehat{P});$$
let us denote this  section by $s$.

\begin{definition} \label{definition K-theory of X,P}
  Let $X$ be a connected proper $G$-space and $P$ a projective unitary
  $G$-equivariant bundle over $X$. The {\it{Twisted $G$-equivariant
  K-theory}} groups of $X$ twisted by $P$ are defined as the  homotopy  groups  of  the  $G$-equivariant  sections
  $$K^{-p}_G(X;P) := \pi_p \left( \Gamma(X;\Fred^{(0)}(\widehat{P}))^{X\rtimes G}, s \right)$$
  where the base point $s= \widehat{I}$ is the section previously constructed.

\end{definition}

\subsection{Additive structure}
There  exists  a  natural  map  
$$\Gamma(X;\Fred^{(0)}(\widehat{P}))^{X\rtimes G} \times \Gamma(X;\Fred^{(0)}(\widehat{P}))^{X\rtimes G} \to \Gamma(X;\Fred^{(0)}(\widehat{P}))^{X\rtimes G}, $$ 
 inducing  an  abelian  group  structure  on  the twisted  equivariant  $K$- theory  groups,  which   we  will  define  below. Consider  for this  the  following  commutative  diagram.
 
 $$\xymatrix{\Fred^{(0)}(\Hg) \times \Fred^{(0)}(\Hg) \ar@{-->}[d]\ar[r]^{f\circ i } & \Fred^{'}(\Hg) \times \Fred^{'}(\Hg) \ar[d]_{\quad \circ \quad} \\\Fred^{(0)}(\Hg) & \Fred^{'}(\Hg) \ar[l]_{f^{-1}\circ r} }  $$
 where  the  vertical  map  denotes  composition. 
As  the  maps  involved  in the  diagram  are  compatible  with the  conjugation  actions  of  the   groups $\UU(\Hg)_{c.g}$, respectively $ \UU(\HH)_{c.g}$ and  $G$,  for  any  projective  unitary $G$-equivariant  bundle  $P$, this  induces  a pointed   map 

$$\Gamma(X;\Fred^{(0)}(\widehat{P}))^{X\rtimes G}, s) \times (\Gamma(X;\Fred^{(0)}(\widehat{P}))^{X\rtimes G}, s )\to(\Gamma(X;\Fred^{(0)}(\widehat{P}))^{X\rtimes G}, s) .$$
Which  defines  an additive structure  in  $K^{-p}_G(X;P)$.

\subsection{Multiplicative structure}\label{sectionmultiplicative}
We  define an associative  product on twisted K-theory.

\begin{equation}\label{defbullet}
K^{-p}_G(X;P)\times K^{-q}_G(X;P')\rightarrow K^{-(p+q)}_G(X;P\otimes P').
\end{equation}

Induced by the map 
$$(A,A')\mapsto A\widehat{\otimes} I+I\widehat{\otimes} A'$$

defined in $\Fred^{0}({\Hg})$, and  $\widehat{\otimes}$ denotes  the  graded  tensor  product, see  \cite{atiyahsegal} page 20 for  more details. We denote this product by $\bullet$.

We will show next that the product above does not depend on the isomorphism classes of the bundles $P$ and $P'$, during the proof we will explain in detail the meaning of the bundle $P\otimes P'$ used above.

\begin{proposition}\label{propprodwelldef}
Let $G$ be a Lie groupoid and let $X$ be a $G$-proper manifold. Consider two isomorphisms $f:P\to Q$ and $g:P'\to Q'$ of $PU(H)$-principal $G$-bundles over $X$. We have a commutative diagram of the form
\begin{equation}
\xymatrix{
K^{-p}_G(X;P)\times K^{-q}_G(X;P')\ar[d]_-{\widetilde{f}\times \widetilde{g}}\ar[r]^-\bullet & K^{-(p+q)}_G(X;P\otimes P')\ar[d]^-{\widetilde{f\otimes g}}\\
K^{-p}_G(X;Q)\times K^{-q}_G(X;Q') \ar[r]_-\bullet &K^{-(p+q)}_G(X;Q\otimes Q')
}
\end{equation}
where the morphisms denoted by $\widetilde{(\cdot)}$ are canonical isomorphisms induced by $f$ and $g$.
\end{proposition}

\begin{proof}
Remember the action of the group $P\UU(\Hg)$ on $\Fred^{(0)}(\Hg)$ by conjugation 
\begin{align*}
	P\UU(\Hg)\times\Fred^{(0)}(\Hg)&\to\Fred^{(0)}(\Hg)\\
	(\varphi,C)&\mapsto \varphi\cdot C=\varphi\circ C\circ\varphi^{-1}.
	\end{align*} 
Consider also the operation defined on Fredholm operators
	
	\begin{align*}
		\Fred^{(0)}(\Hg)\times\Fred^{(0)}(\Hg)&\to\Fred^{(0)}(\Hg)\\
		(A,B)&\mapsto A\sharp B=A\widehat{\otimes}I\oplus I\widehat{\otimes}B.
		\end{align*}
	We have a natural map induced by the graded tensor product
$$P\UU(\Hg)\times P\UU(\Hg)\xrightarrow{\widehat{\otimes}} P\UU(\Hg\widehat{\otimes}\Hg).$$

Let $\widehat{P}$ and $\widehat{P'}$ be the stable projective unitary $G$-bundles over $X$ associated to $P$ and $P'$, whose transitions maps are 
$$\varphi_{\alpha\beta}:U_\alpha\cap U_\beta\to P\UU(\Hg)\text{ for $P$ and }\phi_{\alpha\beta}:U_\alpha\cap U_\beta\to P\UU(\Hg)\text{ for $P'$}.$$ 

We define the stable projective unitary $G$-bundle $$P\UU(\Hg\widehat{\otimes}\Hg)\to \widehat{P}\otimes \widehat{P'}\to X$$whose transitions maps are

 \begin{align*}\varphi_{\alpha\beta}\widehat{\otimes}\phi_{\alpha\beta}:U_\alpha\cap U_\beta&\to P\UU(\Hg\widehat{\otimes}\Hg)\\x&\to\varphi_{\alpha\beta}(x)\widehat{\otimes}\phi_{\alpha\beta}(x).\end{align*}
 
 Now consider the associated bundles $\FredP$ and $\FredP'$. We can define then $\Fred^{(0)}(\widehat{P}\widehat{\otimes}\widehat{P})$ whose transitions maps are
 \begin{align*}
 	(U_\alpha\cap U_\beta)\times\Fred^{(0)}(\Hg\widehat{\otimes}\Hg)&\to\Fred^{(0)}(\Hg\widehat{\otimes}\Hg)\\
 	(x,C)&\to\left(\varphi_{\alpha\beta}\widehat{\otimes}\phi_{\alpha\beta}\right)(x)\cdot C
 	\end{align*}
 	
 The bundle $\Fred^{(0)}(\widehat{P}\widehat{\otimes}\widehat{P})$ is endowed with a multiplication map
 $$\FredP\times\FredP'\stackrel{m}{\longrightarrow}\Fred^{(0)}(P\otimes P'),$$
 defined locally as
 \begin{align*}
 	U_\alpha\times\left(\Fred^{(0)}(\Hg)\times\Fred^{(0)}(\Hg)\right)&\xrightarrow{m} U_\alpha\times\Fred^{(0)}(\Hg)\\
 	(x,A,B)&\mapsto (x,A\sharp B)
 	\end{align*}
 	As for every $\varphi,\phi\in P\UU(\Hg)$ we have that
 	$$(\varphi\cdot A)\sharp(\phi\cdot B)=(\varphi\widehat{\otimes}\phi)\cdot(A\sharp B),$$
 	then $m$ is a well defined map 
 	$$\FredP\times\FredP'\to\Fred^{(0)}(\widehat{P}\otimes \widehat{P'}).$$

We have to show that the following diagram is commutative
	$$\xymatrix{
		\FredP\times\FredP'\ar[rr]^m\ar[d]_-{\bar{f}\times\bar{f}'}&&\Fred(\widehat{P}\otimes\widehat{P'})\ar[d]^{\bar{f}\otimes\bar{f}'}\\
		\Fred^{(0)}(\widehat{Q})\times \Fred^{(0)}(\widehat{Q'})\ar[rr]_-{m_1}&&\Fred^{(0)}(\widehat{Q}\otimes \widehat{Q'})
		}$$

On fibers we have the following
\begin{align*}
	(\bar{f}\otimes\bar{f}')(m(x,A),(x,B))&=(\bar{f}\otimes\bar{f}')(x,A\sharp B)\\
	&=(x,(\bar{f}\otimes\bar{f}')(x)\cdot(A\sharp B))\\&=(x,(f(x)\cdot A)\sharp (f'(x)\cdot B))\end{align*}	
On the other hand
\begin{align*}
	m_1\left((\bar{f}\times\bar{f}')((x,A),(x,B))\right)&=m_1(f(x)\cdot A,f'(x)\cdot B)\\
	&=(x,(f(x)\cdot A)\sharp (f'(x)\cdot B))\end{align*}	
As all above maps are maps of $PU(H)$-principal $G$-bundles then the diagram is commutative.
	\end{proof}


\subsection{Topologies  on  Fredholm  Operators}\label{subsectionTopFred}

In \cite{tuxustacks} a Fredholm picture of twisted K-theory is introduced. Denote by   $\Fred'(\HH)_{s*}$ the space whose elements are the same as $\Fred'(\HH)$ but with the strong $^\ast$-topology on $B(\HH)$.

\begin{definition}\cite[Thm. 3.15]{tuxustacks} \label{definition K-theory Tu-Xu}
  Let $X$ be a connected $G$-proper space and $P$ a projective unitary
  $G$-equivariant bundle over $X$. The {\it{Twisted $G$-equivariant
  K-theory}} groups of $X$ (in the sense of Tu-Xu-Laurent) twisted by $P$ are defined as the  homotopy  groups  of  the  $G$-equivariant  strong${}^\ast$-continuous sections
  $$\IK^{-p}_G(X;P) := \pi_p \left( \Gamma(X;\Fred'({P})_{s^*})^{X\rtimes G}, s \right).$$
  The bundle $\Fred'({P})_{s^*}$ is defined in a similar way as $\Fred'({P})$.

\end{definition}
We will prove that the functors $K_G^\ast(-,P)$ and $\IK_G^\ast(-,P)$ are naturally equivalent.
\begin{lemma}
  The spaces $\Fred'(\HH)$ and $\Fred'(\HH)_{s^*}$ are $PU(\HH)$-weakly homotopy equivalent.
\end{lemma}
\begin{proof}
  The strategy is to prove that $\Fred'(\HH)_{s^*}$ is a representing of equivariant K-theory. The same proof for $\Fred'(\HH)$ in \cite[Prop. A.22]{atiyahsegal} applies. In particular $GL(\HH)_{s^\ast}$ is $G$-contractible because the homotopy $h_t$ constructed in \cite[Prop. A.21]{atiyahsegal} is continuous in the strong$^\ast$-topology and then the proof applies.
\end{proof}

Using the above lemma one can prove that the identity map defines an equivalence between (twisted) cohomology theories $K_G^\ast(-,P)$ and $\IK_G^\ast(-,P)$. Then we have that the both definitions of twisted K-theory are equivalents. Summarizing

\begin{proposition}\label{TXL=BCV}
For every proper $G$-manifold $X$ and every projective unitary $G$-equivariant bundle over $X$. We have an isomorphism
$$K^{-p}_G(X;P)\cong \IK^{-p}_G(X;P).$$
\end{proposition}
\begin{remark}
In order to simplify the notation from now on we denote by $\HH$  a $\IZ_2$-graded separable Hilbert space and we denote by $\Fred^{(0)}({P})$  the bundle $\FredP$.
\end{remark}

\subsection{Relation with the Kasparov external product}

In \cite{tuxustacks} twisted K-theory for Lie groupoids is defined and in Prop. 6.11 of that work this group is described as a KK-group for the case of proper groupoids. 

\begin{proposition}\cite[Prop. 6.11]{tuxustacks}
If $G \rightrightarrows M$ is a proper Lie groupoid and $M/G$ is compact, then for $i=0,1$, there is a natural isomorphism of $K_G^0(M)$-modules $\chi:KK^i_G(C_0(M),B_P)\rightarrow K^i_G(M,P)$, where $B_P$ is certain $C^*$-algebra associated to the twisting $P$.
\end{proposition}

Using the external Kasparov product they define a product
$$\xymatrix{K_G^i(M,P)\otimes K_G^j(M,P')\ar[rr]^(.55){\bullet_{TXL}}&&K_G^{i+j}(M,P\otimes P').}$$

Following ideas from \cite{higson1987} and using the functoriality of both products $\bullet$ and $\bullet_{TXL}$ one can prove that they are the same.

\begin{definition}
\begin{enumerate}
\item
If $\Phi$ is a $KK_G(C_0(M),B_P)$-cycle, we denote by $\Phi_*$ to the homomorphism
\begin{align*}\Phi_*:KK_G(C_0(M),B_0)\to &\  KK_G(C_0(M),B_P)\\
x\mapsto &\ x \bullet_{TXL}\Phi.\end{align*}
and by 
$\Phi^*$ to the morphism 
\begin{align*}\Phi^*:KK_G(C_0(M),B_0)\to &\  KK_G(C_0(M),B_P)\\
	x\mapsto &\ \Phi\bullet_{TXL}x .\end{align*}
\item If $s\in\Gamma^G(\Fred^{(0)}(P))$ we denote by $\overline{s}$ the homomorphism
\begin{align*}\overline{s}:K_G^i(X)\to\ &K_G^i(X,P)\\
[f]\mapsto&\ [s\bullet f].\end{align*}
\end{enumerate}
\end{definition}
\begin{proposition}\label{productsequiv}
If $\Phi\in KK^i_G(C_0(M),B_P)$ and $\Psi\in KK^i_G(C_0(M),B_{P'})$, then $\chi(\Phi\bullet_{TXL}\Psi)=\chi(\Phi)\bullet \chi(\Psi)$.
\end{proposition}
\begin{proof}For this proof we denote by $1_{C_0(M)}$ the multiplicative identity of $K_G(M)$
 \begin{align*}
 \chi(\Phi\bullet_{TXL}\Psi)&= \chi(\Phi_*(1_{C_0(M)})\bullet_{TXL}\Psi^*(1_{C_0(M)}))\\
 &=\chi(\Phi_*(\Psi^*(1_{C_0(M)})))\text{ ($1_{C_0(M)}$ is the multiplicative identity)}\\
 &=\overline{\chi(\Phi)}\left(\overline{(\chi(\Psi))}(\chi(1_{C_0(M)}))\right)\text{(the naturality of $\chi$)}\\
 &=\overline{\chi(\Phi)}(1_{C_0(M)})\bullet \overline{\chi(\Psi)}(1_{C_0(M)})\\
 &=\chi(\Phi)\bullet \chi(\Psi).
  \end{align*}
\end{proof}
The above result implies that both products are the same modulo the equivalence $\chi$. In particular we have:

\begin{corollary}\label{prodFredassoc}
The product $\bullet$ defined in (\ref{defbullet}) above is associative.
\end{corollary}

\section{Thom isomorphism}

Let $G\rightrightarrows G_0$ be a Lie groupoid and $P$ a twisting. Consider a $G$-oriented vector bundle $E\longrightarrow X$. In particular since we will assume that $G$ acts properly on $P$ and on $E$, we can assume $E$ admits a $G$-invariant metric, see for instance \cite{PPT} proposition 3.14 and \cite{HF} theorem 4.3.4.. As explained in \cite{CaWangBC} appendix A (especially proposition A.3), in this situation there is a natural isomorphism
$$Th:{\bf K_G^*(X,P)}\rightarrow {\bf K_G^{*-rank(E)}(E,\pi^*(P\otimes \beta_E))}$$ 
where ${\bf K_G^*(X,P)}$ stands for the K-theory of the twisted groupoid $C^*$-algebra 
$C^*_r(X\rtimes G,P)$ and where $\beta_E$ is the orientation $G$-twisting over $E$ defined in example (\ref{obundle}) in \ref{example} above.
The fact that it is indeed the Thom isomorphism comes from the functoriality and the naturality with respect to the Kasparov products of the Le Gall's descent construction \cite{LeGall} theorem 7.2. This is explained in details in the appendix cited above or in \cite{Elkaioum} in the context of real groupoids (the same arguments apply in the complex case). 

Now, in \cite{tuxustacks} theorem 3.14 the authors prove that for proper Lie groupoids the groups
${\bf K_G^*(X,P)}$ and $\IK^{-p}_G(X;P)$ are naturally isomorphic. We thus obtain, by proposition \ref{TXL=BCV}, the Thom isomorphism

$$Th:K_G^*(X,P)\rightarrow K_G^{*-rank(E)}(E,\pi^*(P\otimes \beta_E)).$$ 

It is possible however to construct the Thom isomorphism directly in the Fredholm picture of the twisted K-theory (whenever the respective action groupoids are proper), we will perform this construction for  the benefit of the reader.

\subsection*{The   spin  representation  and  twisted  K-Theory}

Let  $n$  be  an  even  natural number. 

Let  $\IR^n$  denote  the  euclidean, $n$-dimensional  vector  space  denoted  with  the  euclidean  scalar product. 

The Clifford  algebra $\Cliff(\IR^n)$ 
is  defined  as  the  complexification  of the   quotient  of  the  tensor algebra   $ T\IR^ n= \underset{j=0}{\overset{\infty} {\bigotimes}} \IR^ n$ by  the  two-sided  ideal  defined  by  elents  of  the  form $ x\otimes x-\langle x,x\rangle$, where  $\langle\quad \rangle $ denotes  the   euclidean scalar  product. 
  
It is  generated  as  $\IC$-algebra   by  the  elements of a an  orthogonal  basis  $e_i$ of  $ \IR^ n$     with   the  relations $e_i \cdot e_j = -2 \delta_{i,j}$. 

The  algebra $\Cliff(\IR^n)$ is isomorphic  as a  vector  space   to   the exterior algebra
$\Lambda^*(\IR^n)=\underset{j=0}{\overset{n} {\bigoplus}}\Lambda^j \IR^n$  \cite{lawsonmichelson},  Proposition  1.3  in page  10,  in  particular, it   has   complex  dimension  $2^n $ .

The  map  given by Clifford  multiplication  with the  element  $ e_1\cdot, \ldots, \cdot  e_n$ defines  a linear  operator  on  $ \Cliff(\IR^n)$. The   Clifford  algebra  then  decomposes  as a  vector  space $\Cliff(\IR^n)= S^+ \oplus S^-$,  where $S^+$  is  the eigenspace  associated to $+1$  and $S^-$  is  the one  associated  with  $-1$. An  element  in  $S^+$  is  called  even,  an  element  in  $S^-$ is  said  to  be  odd. 

The  group  $ \Spin(\IR^n) $  consists  of  the  multiplicative  group  of even units in the  Clifford algebra, in  symbols $ \Spin(\IR^n)= \Cliff(\IR^n)^{*} \cap S^+.$

The group  $\Spin(\IR^n)$
is the universal   covering  of  the  special  orthogonal group $\SO(n)$.
The  map 
$$1 \to  \IZ_ 2 \to  \Spin(\IR^n)\to SO(n)\to 1$$   

is  a  model  for  the  universal  central  extension of  $ SO(n)$.

This  extension  is classified  by   the  nontrivial  class $\tau\in H^2(SO(n), S^1)\cong \IZ_2$.

The  group $\Spin(\IR^n)$  has  a complex  linear  representation $\rho: \Spin(\IR^n)\to U(2^n)$,  given  by  the  identification  of  $ \operatorname{\IC liff}(\IR^n)=\Cliff(\IR^n)\otimes\IC$  with the  complex  vector  space  of  dimension $2^n$ as  an algebra,  and the linear  operator  given  by   $ \rho(x):v\mapsto x^{-1} vx $. 

The  representation  $\rho$ gives  rise  to a   continuous  group  homomorphism  $ \beta$  as in the  following  diagram:  
$$ \xymatrix {1 \ar[r ]&\ar[r] S^1 \ar[d]\ar[r] & \Spin^c(\IR^n) \ar[r] \ar[d]^{\rho} & SO(n) \ar[r]\ar[d]^{\beta} &  1\\  1 \ar[r] & S^1 \ar[r] &U(2^n) \ar[r] & PU(2^n)\ar[r] & 1 }$$

\begin{definition}
The spin  representation  is  the  homomorphism $\beta : \SO(n )\to P \UU( \HH)$  
\end{definition}
   
\begin{remark}
Let $n$ be a even positive integer. Consider a  proper oriented $G$-vector bundle $E\xrightarrow{\pi}X$ over a  proper $G$-manifold $X$. We can suppose that the chart data is given by a generalized morphism $$\xymatrix{X\rtimes G\ar@{-->}[r]^{O_E}& \SO(n)}.$$
Composing the generalized morphism $O_E$ with the spin representation $\beta$ we obtain a twisting $\xymatrix{\beta_E:X\rtimes G\ar@{-->}[r]&P\UU(\HH)}$, called the \textit{orientation twisting}.

\end{remark}
 We will construct now the Thom class in the Fredholm picture. If $X$ is a proper $G$-manifold, by Theorem 2.3 in \cite{zung} for every $x\in X$ there is a open neighbourhood $U$ of $x$ contractible to the orbit of $x$ in $X\rtimes G$ with action of the isotropy group $G_x$ such that there is a Lie groupoid isomorphism 
 $$(X\rtimes G)\mid U\cong U\rtimes G_x.$$
 
 We have an isomorphism \begin{equation}\label{thomtrivial}K^{-n}_{G_x}(U,\beta_E\mid_U)\cong R_{S^1}(\widetilde{G_x}),\end{equation}
 where $\widetilde{G_x}$ is the  $S^1$-central extension of $G_x$ associated to the twisting $\beta_E\mid_U$. On the other hand, $E\mid_{\{x\}}$ is a real representation of $G_x$, since it can be  viewed as a homomorphism $\eta_x:G_x\rightarrow SO(n)$. The composition $\beta\circ\eta_x:G\rightarrow P\UU(\HH)$ is a projective representation and its isomorphism class determines an  element of $R_{S^1}(\widetilde{G_x})$. Using the identification \ref{thomtrivial}, it can be viewed as an element of $K^{-n}_{G_x}(U,\beta_E\mid_U)$. We denote this element by $\lambda_{-1}^U$.

Taking a covering of $X\rtimes G$ one can see that these local elements are the same on intersections. The  local  trivializations  define a global element 
$$[\lambda_{-1}^E]\in K_G^{-n}(X,\beta_E),$$
we call it the \textit{Thom class}.

Given $s\in\Gamma^G(P\times_{P\UU(\HH)}\Fred^{(0)}(\HH))$, where $P\rightarrow X$ is a twisting, we define the Thom isomorphism

\begin{align*}Th:K_G^*(X,P)\to&    K_G^{*-n}(E,\pi^*(P\otimes \beta_E))\\
            [s]          \mapsto& [e\mapsto s(\pi(e))\bullet \lambda_{-1}^E(\pi(e))].\end{align*}

When the vector bundle $E$ is odd dimensional, using the classic suspension isomorphism and the previous Thom isomorphism for $E\oplus \IR$, one gets as well a Thom isomorphism as above.

Since the Thom isomorphism is natural with respect to the Kasparov product we can summarize the discussion above in the following statement.

  \begin{theorem}\label{Thomring}[Thom isomorphism]
  With notations as above, there is a natural isomorphism $$Th:K_G^*(X,P)\rightarrow K_G^{*-rank(E)}(E,\pi^*(P\otimes \beta_E))$$ 
 which gives the Thom isomorphism. If $E$ is a $Spin^c$ $G-$ vector bundle the above isomorphism is compatible with the external $\bullet$ product.
  \end{theorem} 
  
\section{Pushforward and Pullback Maps }

\subsection{Pushforward}\label{pushforwardsection}




In this section we will recall how to define the pushforward morphism associated to any smooth $G$-map $f:X\to Y$ between to $G$- manifolds, definition 4.1 in \cite{CaWangBC}. For the purpose of this paper we will perform the construction in the case of K-oriented maps. By this we mean that the bundle $T^*X\oplus f^*(TY)$ admits a $Spin^c-$structure.

The difference in the present construction with respect to ref.cit. is that we will not make reference to $C^*$-algebras and we will perform the construction using the Fredholm picture of the twisted K-theory, in particular the construction below works only for $G$-proper manifolds.

We will need to state some general statements about groupoids that will simplify the particular constructions we are interested in. 

\begin{lemma}\label{excisiongrpdK}
Let $G\rightrightarrows G_0$ be a proper Lie groupoid together with a twisting P. Let $H\rightrightarrows H_0$ be a proper Lie saturated closed subgroupoid.
\begin{enumerate}
\item There is a canonical restriction morphism
\begin{equation}
K_{G}^{-p}(G_0,P)\to K_{H}^{-p}(H_0,P|_{H_0})
\end{equation}
\item Suppose $G$ decomposes as the union of two saturated proper subgroupoids $G=H\sqcup H'\rightrightarrows H_0\sqcup H'_0$ with $H$ closed subgroupoid. There is a long exact sequence
\begin{equation}
\tiny{
\xymatrix{
\ar[r]&K_{H'}^{-p}(H'_0,P|_{H'_0})\ar[r]&K_{G}^{-p}(G_0,P)\ar[r]&K_{H}^{-p}(H_0,P|_{H_0})\ar[r]&K_{H'}^{-p-1}(H'_0,P|_{H'_0})\ar[r]&
}}
\end{equation}
\end{enumerate}
\end{lemma}

\begin{lemma}\label{homotopygrpdK}
Let $G\rightrightarrows G_0$ be a proper Lie groupoid together with a twisting P, consider the product groupoid $G\times (0,1]\rightrightarrows G_0\times (0,1]$ with the pullback twisting $P_{(0,1]}$. For every $p\in \mathbb{Z}$
$$K_{G\times (0,1]}^{-p}(G_0\times (0,1],P_{(0,1]})=0.$$
\end{lemma}

The two previous lemmas are classic in the $C^*$-algebraic context, {\it i.e.,} once we use that the isomorphism between the twisted K-theory with the $C^*$-picture and the twisted K-theory with the Fredholm picture (theorem 3.14 \cite{tuxustacks}).

The following result is an immediate consequence of lemmas \ref{excisiongrpdK} and \ref{homotopygrpdK} above.

\begin{proposition}
Given an immersion of proper Lie groupoids $G_1\stackrel{\varphi}{\rightarrow}G_2$  and a twisting $\alpha$ on $G_2$, consider the twisted deformation groupoid $(G_\varphi,P_\alpha)$ of section \ref{defgrpds} (propositions \ref{HSimmer} and \ref{deftwistimm}). The morphism in K-theory induced by the restriction at zero,
\begin{equation}
\xymatrix{
K_{G_\varphi}^{-p}(G_{\varphi}^{(0)},P_{\varphi})\ar[rr]^-{e_0}&&
K_{G_2}^{-p}(G_{2}^{(0)},P_2)
}
\end{equation}
is an isomorphism.
\end{proposition}

\begin{definition}[Index associated to a groupoid immersion]
Given an immersion of proper Lie groupoids $G_1\stackrel{\varphi}{\rightarrow}G_2$  as above and a twisting $\alpha$ on $G_2$, we let
\begin{equation}
Ind_\varphi:K_{G_1^N}^{-p}((G_{1}^{(0)})^N,P_1^N)\to K_{G_2}^{-p}(G_{2}^{(0)},P_2)
\end{equation}
to be the morphism in K-theory given by $Ind_\varphi:=e_1\circ e_{0}^{-1}$.
\end{definition}

We are ready to define the shriek map. Let $G\rightrightarrows G_0$ be a Lie groupoid together with a twisting $P$. Let $X,Y$ be two $G$-proper manifolds and let $f:X\to Y$ be a smooth $G$-map with $T^*X\oplus f^*TY$ a $G$-$Spin^c$ vector bundle that we will assume in a first time to have even rank. We will also assume the moment maps $X\to G_0$ and $Y\to G_0$ to be submersions, then $T^*X\oplus f^*TY$ being $Spin^c$ is equivalent to 
$V_f:=T_v^*X\oplus f^*T_vY$ being $Spin^c$.
The shriek morphism
\begin{equation}
\xymatrix{
f!:K_{G}^{-p}(X,P_X) \ar[r]^-{f_!}&K_{G}^{-p-d_f}(Y,P_Y),
}
\end{equation}
where $d_f:=rank\, V_f$,
will be given as the composition of the following three morphism

{\bf I.} The twisted $G$-equivariant Thom isomorphism 

\begin{equation}\label{Step1map}
\xymatrix{
K_{G}^{-p}(X,P_X)\ar[r]^-{T}_-{\cong}&K_{G}^{-p-d_f}(T_v^*X\bigoplus f^*T_vY,P_{V_f}).
}
\end{equation}

{\bf II.}  
We consider now the index morphism 

\begin{equation}\label{Step2map}
\xymatrix{
K_{(T_vX\bigoplus f^*T_vY)\rtimes G}^{-p-d_f}(f^*T_vY,P)\ar[r]^-{Ind}& K_{ f^*T_vY\rtimes (T_vX\rtimes G)}^{-p-d_f}(f^*T_vY,P)
}
\end{equation}
associated to the immersion 
$$f^*T_vY\rtimes G\longrightarrow f^*T_vY\rtimes (T_vX\rtimes G)$$
given by the product of the identity in $G$ and the inclusion of the units $f^*T_vY$ in the groupoid $f^*T_vY\rtimes T_vX$.

{\bf III.}
Consider the groupoid immersion
\begin{equation}\label{tildef}
\xymatrix{
X\rtimes G\ar[rr]^-{\tilde{f}} && (Y\times_{G_0}(X\times_{G_0} X))\rtimes G,
}
\end{equation}
where $\tilde{f}:=(f\times \triangle)\times Id_G$. Then the induced deformation groupoid is 
$$G_f\rtimes G$$
where
$$G_f\rightrightarrows G_{f}^{(0)}$$
is the groupoid given by
\begin{equation}\label{Gf}
G_f:=f^*(T_vY)\rtimes T_vX\times \{0\}\bigsqcup Y\times_{G_0}
(X\times_{G_0} X) \times (0,1]\, \text{ and }
\end{equation}
\begin{equation}
G_{f}^{(0)}=f^*T_vY\times \{0\}\bigsqcup Y\times_{G_0} 
X\times (0,1]
\end{equation}
Notice that $Y\times_{G_0}(X\times_{G_0} X)$ and $Y$ are Morita equivalent groupoids with Morita equivalence the canonical projection.

Let $\alpha_f$ the twisting on $G_f\rtimes G$ given by proposition \ref{deftwistimm}. It is immediate to check that $\alpha_f|_{(f^*(T_vY)\rtimes T_vX)\rtimes G}
=\pi_{f^*T_vY\rtimes T_vX}^*\alpha$.

We can hence consider the twisted deformation index morphism associated to $(G_f\rtimes G,\alpha_f)$ :

\begin{equation}\label{Step3map}
\xymatrix{
K_{f^*T_vY\rtimes (T_vX\rtimes G)}^{-p-d_f}(f^*T_vY,P)\ar[r]^-{Ind_f}&
K_{(Y\times_{G_0}(X\times_{G_0} X))\rtimes G}^{-p-d_f}(Y\times_{G_0}X,P)\ar[d]^-{\mu}_-{\cong}\\&
K_{G}^{-p-d_f}(Y,P)
}
\end{equation} 

For composing \ref{Step1map} with \ref{Step2map} remember that by the Fourier isomorphism proved in proposition 2.12 in \cite{CaWangAdv} and by theorem 3.14 in \cite{tuxustacks} we have an isomorphism

$$K_{G}^{*}(T_v^*X\bigoplus f^*T_vY,P_{V_f})\approx K_{(T_vX\bigoplus f^*T_vY)\rtimes G}^{*}(f^*T_vY,P).$$

We can now give the following definition:

\begin{definition}[Pushforward morphism for twisted $G$-manifolds]
Let $X,Y$ be two manifolds and $f:X\longrightarrow Y$ a smooth map as above. Under the presence of a twisting P on $G$ we let
\begin{equation}
\xymatrix{
K_{G}^{-p}(X,P_X) \ar[r]^-{f_!}&K_{G}^{-p-d_f}(Y,P_Y)
}
\end{equation}
to be the morphism given by the composition of the three morphisms described above, \ref{Step1map} followed by \ref{Step2map} followed by \ref{Step3map}.
\end{definition}

One of the main results is that the pushforward maps is compatible with the product:

\begin{proposition}\label{shriekring}
Let $G\rightrightarrows G_0$ be a Lie groupoid. Let $X,Y$ be two $G$-proper manifolds and let $f:X\to Y$ be a $G$-smooth K-oriented map with $T_v^*X\bigoplus f^*T_vY$ of even rank. Let $P$ and $P'$ two $G$-$PU(H)$-principal bundles over $G_0$. The following diagram is commutative:
\begin{equation}\label{shriekringdiag}
\xymatrix{
K^{-p}_G(X;P)\times K^{-q}_G(X;P')\ar[d]_-{f!\times f!}\ar[r]^-\bullet & K^{-(p+q)}_G(X;P\otimes P')\ar[d]^-{f!}\\
K^{-p}_G(Y;P)\times K^{-q}_G(Y;P')\ar[r]_-\bullet & K^{-(p+q)}_G(Y;P\otimes P')
}
\end{equation}
\end{proposition}
\begin{proof}
By definition $f!$ is constructed by means of a Thom isomorphism and of two deformation indices. These indices are at their turn constructed by restriction (or evaluation) morphisms. To conclude the proof one has only to observe that restrictions are obviuosly compatible with the product together with the fact that Thom is also compatible with the product, see \ref{Thomring}.
\end{proof}

\subsection{The Pullback:}
Let $A\stackrel{h}{\longrightarrow}B$ be a smooth $G$-equivariant map ($A, B$ $G$-proper manifolds). Suppose we have a $PU(H)$-principal $G$-bundle $P$ over $G_0$. We are going to consider, for every $q\in \mathbb{N}$, the pullback

\begin{equation}
h^*: K_{G}^{-q}(B,P_B)\longrightarrow K_{G}^{-q}(A,P_A)
\end{equation}
given as follows: If $\gamma:S^q \to \Gamma(B,Fred(\hat{P_B}))^G$ is a continuous map with $\gamma(*)=s$ one let 
$$h^*\gamma:S^q\to\Gamma(A,Fred(\hat{P_A}))^{A\rtimes G}$$
to be given by 
$$(h^*\gamma)(z)(a):=\gamma(z)(h(a)),$$
it is then classic to show that it induces a map between the homopoty classes.

More generally we will need a pullback map associated to a $G$-equivariant Hilsum-Skandalis map. We explain next what do we mean by this.

Consider a Lie groupoid $H_A\rightrightarrows A$, we say that it is a $G$-groupoid if $G$ acts on $H_A$, on $A$ and the source and target maps of $H_A$ are $G-$equivariant. Under this situation we might form the semi-direct product groupoid 
$$H_A\rtimes G\rightrightarrows A.$$
Suppose now that we have two $G$-proper (all the actions are required to be proper) Lie groupoids  $H_A$ and $H_B$ together with a generalized morphism $h:H_A--->H_B$ between them, that is, suppose we are given a $H_B$-principal bundle $P_h$ over $H_A$, putting this in a diagram:
\[
\xymatrix{
H_A\ar@<.5ex>[d]\ar@<-.5ex>[d]&P_h\ar[rd]^-{s_h}\ar[ld]_-{t_h}
&H_B\ar@<.5ex>[d]\ar@<-.5ex>[d]\\
A&&B
}
\]
We are going to consider, for every $q\in \mathbb{N}$, the pullback

\begin{equation}
h^*: K_{H_B\rtimes G}^{-q}(B,P_B)\longrightarrow K_{H_A\rtimes G}^{-q}(A,P_A)
\end{equation}
given as follows: If $\gamma:S^q \to \Gamma(B,Fred(\hat{P_B}))^{H_B\rtimes G}$ is a continuous map with $\gamma(*)=s$ one let 
$$h^*\gamma:S^q\to\Gamma(A,Fred(\hat{P_A}))^{H_A\rtimes G}$$
to be given by 
$$(h^*\gamma)(z)(a):=\gamma(z)(b)$$
where $b=s_h(v)$ for some $v\in t_h^{-1}(a)$. One proves using the invariance of $\gamma$ together with the identification
$P_h\times_{H_B\rtimes G}Fred(\hat{P_B})=Fred(\hat{P_A})$ that the definition of $h^*\gamma$ does not depend on the choice of $v$.

The use of the Fredholm picture for K-theory allows to give a very classic definition for the pullback map and to adapt word by word the classic proofs that it is well defined and the following naturality result:

\begin{lemma}\label{pullnatural}
The pullback is natural. The following properties hold:
\begin{enumerate}
\item $Id^*=Id$
\item $(h_2\circ h_1)^*=h_1^*\circ h_2^*$
\end{enumerate}
\end{lemma}

The following proposition is also an example of the convenience of the Fredholm model for $K$-theory, indeed, its $C^*$-algebraic analog is much harder to prove and corresponds to Le Gall's pullback naturality with respect to Kasparov's products. Of course Le Gall's results are more general and apply to more complicated situations (see remark below). Here we only need for the moment the proper case. We state the result.

\begin{proposition}\label{pullring}
Let $G\rightrightarrows G_0$ be a Lie groupoid. Let $A,B$ be two $G$-proper manifolds and let $h:A\to B$ be a $G$-smooth K-oriented map. Let $P$ and $P'$ two $G$-$PU(H)$-principal bundles over $G_0$. The following diagram is commutative:
\begin{equation}\label{pullringdiag}
\xymatrix{
K^{-p}_G(B;P)\times K^{-q}_G(B;P')\ar[d]_-{h^*\times h^*}\ar[r]^-\bullet & K^{-(p+q)}_G(B;P\otimes P')\ar[d]^-{h^*}\\
K^{-p}_G(A;P)\times K^{-q}_G(A;P')\ar[r]_-\bullet & K^{-(p+q)}_G(A;P\otimes P')
}
\end{equation}
\end{proposition}

\begin{proof}
Remember that in the notations above, for a $G$-proper manifold $X$, $K^{-p}_G(A;P)$ means we are considering equivariant twisted $K-$theory of $X$ with respect to the bundle $P_X:=\pi_X^*P=X\times_MP$ where $\pi_X:X\to M$ is the momentum map of the $G$-action. 
In particular if $h$ is as above, there is an induced bundle map $\widetilde{h}:P_A\to P_B$ given by $h$ in the direction of $A$ and the identity in the direction of $P$. The same of course applies for the bundles used below (all the bundles come from M by pullback). In particular there is a trivially commutative diagram of bundles
\begin{equation}
\xymatrix{
Fred^{(0)}(\hat{P}_B)\times Fred^{(0)}(\hat{P'}_B)\ar[r]^-m&Fred^{(0)}((\hat{P}\otimes \hat{P'})_B)\\
Fred^{(0)}(\hat{P}_A)\times Fred^{(0)}(\hat{P'}_A)\ar[u]^-{\widetilde{h}}\ar[r]_-m&Fred^{(0)}((\hat{P}\otimes \hat{P'})_A)\ar[u]_-{\widetilde{h}}
}
\end{equation}
where $m$ is the map defined in proposition \ref{propprodwelldef} to properly define the product $\bullet$. By definition of the pullback, the commutativity of the above diagram induces the commutativity of diagram (\ref{pullringdiag}).
\end{proof}

\begin{remark}[On Le Gall's descent functors]
The definition of the pullback above recalls Le Gall's pullback construction on the untwisted case which generalizes Kasparov descent morphisms. The simplicity of our construction is due to the fact that we are only dealing with the proper action case. In the general case is certainly possible to adapt Le Gall's to $S^1$-central extensions and then to apply it to the Twisted K-theory case. In particular we could prove the proposition above using theorem 7.2 in \cite{LeGall} which states the naturality of the pullback with respect to the Kasparov product. We prefered however to give a direct proof since in the proper case it is possible.
\end{remark}

The main property of this section is the naturality of the pushforward maps with respect to pullbacks, this is one of the new and one of the main key technical result in this paper.

\begin{proposition}\label{shrieknatural}
Let $G\rightrightarrows G_0$ be a Lie groupoid together with a twisting $P$. Suppose we have a commutative diagram of $G$-smooth K-oriented maps between $G$-proper manifolds
$$
\xymatrix{
A\ar[r]^-g\ar[d]_-p&A'\ar[d]^-q\\
B\ar[r]_-f&B'
}
$$
Then we have the following equality between K-theory morphisms 
$$g!\circ p^*=q^*\circ f!$$
\end{proposition}
\begin{proof}
We have to show that the following diagram is commutative
\begin{equation}\label{totaldiag}
\xymatrix{
K_{G}^{*}(A,P_A)\ar[r]^-{g!}&K_{G}^{*}(A',P_{A'})&\\
K_{G}^{*}(B,P_B)\ar[u]^-{p^*}\ar[r]_-{f!}&K_{G}^{*}(B',P_{B'})\ar[u]_-{q^*}
}
\end{equation}
We will split the above diagram in four commutative diagrams:

{\bf Diagram I.} 
Consider the following commutative diagram of groupoid morphisms which are equivariant with respect to the $G$-action:
$$
\xymatrix{
A'\times_{G_0}(A\times_{G_0}A)\ar[d]_-{q\times \triangle(p)}\ar[rr]^-{Id_{A'}\times Pr_{G_0}}&&A'\times_{G_0}G_0\ar[d]^-{q\times Id_{G_0}}\\
B'\times_{G_0}(B\times_{G_0}B)\ar[rr]^-{Id_{B'}\times Pr_{G_0}}&&B'\times_{G_0}G_0
}
$$
Once identifying $A'\times_{G_0}G_0$ with $A'$ (and respectively for $B'$) we have that $Id_{A'}\times Pr_{G_0}$ induces the Morita equivalence of groupoids between $A'\times_{G_0}(A\times_{G_0}A)$ and $A'$ with inverse a Hilsum-Skandalis isomorphism that induces the isomorphism $\mu$ in K-theory. Hence the diagram above induces the following commutative diagram in $K$-theory:
\begin{equation}\label{DiagI}
\xymatrix{
K_{G\ltimes(A'\times_{G_0}(A\times_{G_0}A))}^{*}(A'\times_{G_0}A,P_{A'\times_{G_0}A})\ar[rr]^-{\mu}_-{\approx}\ar@{}[rrd]|-{{\bf I}}&&K_{G}^{*}(A',P_{A'})\\
K_{G\ltimes(B'\times_{G_0}(B\times_{G_0}B))}^{*}(B'\times_{G_0}B,P_{B'\times_{G_0}B})\ar[u]^-{(q\times_{G_0}\triangle(p))^*}\ar[rr]_-{\mu}^-{\approx}&&K_{G}^{*}(B',P_{B'})\ar[u]_-{q^*}
}
\end{equation}
{\bf Diagram II.} 
Remember the $G$-groupoid immersions 
$$A\stackrel{g\times \triangle}{\longrightarrow}A'\times_{G_0}(A\times_{G_0}A)$$
and 
$$B\stackrel{f\times \triangle}{\longrightarrow}B'\times_{G_0}(B\times_{G_0}B)$$
used above to construct the deformation indices ( see (\ref{tildef}) and \ref{Step3map}). They fit in the following commutative diagram of $G$-morphisms:
\begin{equation}
\xymatrix{
A'\times_{G_0}(A\times_{G_0}A)\ar[rr]^-{q\times \triangle(p)}&&B'\times_{G_0}(B\times_{G_0}B)\\
A\ar[u]^-{g\times \triangle}\ar[rr]_-p&&B\ar[u]_-{f\times \triangle}
}
\end{equation}

By the functoriality of the deformation to the normal cone we have a morphism of $G$-groupoids (see (\ref{Gf}) for notations)
$$
\xymatrix{
G_g\ar@<.5ex>[d]\ar@<-.5ex>[d]\ar[rr]^-{\widetilde{q\times \triangle(p)}}&&G_f\ar@<.5ex>[d]\ar@<-.5ex>[d]\\
G_g^{(0)}\ar[rr]_-{(\widetilde{q\times \triangle(p)})_0}&&G_f^{(0)}
}
$$
whose restriction at $t=1$ gives $q\times \triangle(p)$ and whose restriction at $t=0$ gives $d^vp\ltimes d^vq:T_vA\ltimes g^*T_vA'\to T_vB\ltimes f^*T_vB'$ as a morphism of $G$-groupoids where $d^vp$ (resp. $d^vq$) stands for the derivative in the tangent vertical direction. Since pullbacks obviuosly commutes with restrictions we have the following commutative diagram
\begin{equation}\label{DiagII}
\xymatrix{
K_{G\ltimes(T_vA\ltimes g^*T_vA')}^{*}(g^*T_vA',P_{g^*T_vA'})\ar[rr]^-{Ind_{\tilde{g}}}\ar@{}[rrd]|-{{\bf II}}&&K_{G\ltimes(A'\times_{G_0}(A\times_{G_0}A))}^{*}(A'\times_{G_0}A,P_{A'\times_{G_0}A})\\
K_{G\ltimes(T_vB\ltimes f^*T_vB')}^{*}(f^*T_vB',P_{f^*T_vB'})\ar[rr]_-{Ind_{\tilde{f}}}\ar[u]^-{(d^vp\ltimes d^vq)^*}&&K_{G\ltimes(B'\times_{G_0}(B\times_{G_0}B))}^{*}(B'\times_{G_0}B,P_{B'\times_{G_0}B})\ar[u]_-{(q\times_{G_0}\triangle(p))^*}
}
\end{equation}

{\bf Diagram III.} 
The groupoid morphism (equivariant w.r. to $G$)
$$d^vp\ltimes d^vq:T_vA\ltimes g^*T_vA'\to T_vB\ltimes f^*T_vB'$$
induces (again by functoriality of the deformation to the normal cone) a $G$-groupoid morphism between the respective tangent groupoids
$$(d^vp\ltimes d^vq)^{tan}:(T_vA\ltimes g^*T_vA')^{tan}\to (T_vB\ltimes f^*T_vB')^{tan}$$ whose restriction at $t=1$ gives $d^vp\ltimes d^vq$ and whose restriction at zero gives $d^vp\oplus d^vq:T_vA\oplus g^*T_vA'\to T_vB\oplus f^*T_vB'$. For the same reason as diagram II we have the following commutative diagram in K-theory:

\begin{equation}\label{DiagIII}
\xymatrix{
K_{G\ltimes(T_vA\oplus g^*T_vA')}^{*}(g^*T_vA',P_{g^*T_vA'})\ar@{}[rrd]|-{{\bf III}}\ar[rr]^-{Ind}&&K_{G\ltimes(T_vA\ltimes g^*T_vA')}^{*}(g^*T_vA',P_{g^*T_vA'})\\
K_{G\ltimes(T_vB\oplus f^*T_vB')}^{*}(f^*T_vB',P_{f^*T_vB'})\ar[u]^-{(d^vp\oplus d^vq)^*}\ar[rr]_-{Ind}&&K_{G\ltimes(T_vB\ltimes f^*T_vB')}^{*}(f^*T_vB',P_{f^*T_vB'})\ar[u]_-{(d^vp\ltimes d^vq)^*}
}
\end{equation}

{\bf Diagram IV.} 
The commutativity of the following diagram follows from the naturality of Thom isomorphism:
\begin{equation}\label{DiagIV}
\xymatrix{
K_{G}^{*}(A,P_A)\ar[rr]^-{Thom}_-{\approx}&&K_{G\ltimes(T_vA\oplus g^*T_vA')}^{*}(g^*T_vA',P_{g^*T_vA'})\\
K_{G}^{*}(B,P_B)\ar[u]^-{p^*}\ar[rr]_-{Thom}^-{\approx}&&K_{G\ltimes(T_vB\oplus f^*T_vB')}^{*}(f^*T_vB',P_{f^*T_vB'})\ar[u]_-{(d^vp\oplus d^vq)^*}
}
\end{equation}

By definition, diagram (\ref{totaldiag}) decomposes, with the previous diagrams, in the following form:
$$
\xymatrix{
\ar[r]\ar@{}[rd]|-{{\bf IV}}&\ar[r]\ar@{}[rd]|-{{\bf III}}&\ar[r]\ar@{}[rd]|-{{\bf II}}&\ar[r]\ar@{}[rd]|-{{\bf I}}&\\
\ar[r]\ar[u]&\ar[u]\ar[r]&\ar[u]\ar[r]&\ar[u]\ar[r]&\ar[u]
}
$$
and hence it is commutative.
\end{proof}

\section{Product on the Twisted K-homology for Lie groupoids}\label{twistsection}


The pushforward functoriality theorem (thm. 4.2 in \cite{CaWangBC}) allows us to give the following definition:

\begin{definition}[Twisted geometric K-homology fo Lie groupoids \`a la Connes]\label{twistkhom}
Let $G \rightrightarrows M$ be a Lie groupoid with a twisting 
$\alpha$. Take $P_\alpha$ a $PU(H)$-principal $G$-bundle over $M$ representing $\alpha$. By the "Twisted geometric K-homology group" associated to 
$(G,\alpha)$ we mean the abelian group denoted by $K_{*}^{geo}(G,\alpha)$ with generators the cycles $(X,x)$ where
\begin{itemize}
\item[(1)]   $X$ is a smooth  co-compact $G$-proper manifold,
\item[(2)]   $\pi_X:X\to M$ is  the smooth momentum map which supposed to be a K-oriented  submersion and
\item[(3)] $x\in K_G^{-p}(X, P_X)$ for some $p\in \mathbb{N}$,
\end{itemize}
and relations given by
\begin{equation}
(X,x)\sim (X',g_!(x))
\end{equation}
where $g:X\to X'$ is a smooth $G$-equivariant map.

The group above depends on the choice $P_\alpha$, for different isomorphic bundles the respective groups are isomorphic as well, we will discuss this on the last section.

The group defined above admits a $\mathbb{Z}_2$-gradation
$$K_{*}^{geo}(G,\alpha)=K_{0}^{geo}(G,\alpha)\bigoplus K_{1}^{geo}(G,\alpha).$$
where $K_{j}^{geo}(G,\alpha)$ is the subgroup generated by cycles $(X,x)$ for which $T_vX$ has rank congruent to $j$ modulo 2.
\end{definition}

We will now describe a product between two cycles with possibly different twistings by using the product structure defined in previous sections.

{\bf The product of two cycles:}
Let $P$ and $Q$ two $PU(H)$-principal bundles on $G$. Let $(X,x)$ with $x\in K_{G}^{-p}(X,P_X)$ and $(Y,y)$ with $y\in K_{G}^{-q}(Y,Q_Y)$ we put 
\begin{equation}
(X,x)\cdot (Y,y):=(X\times_{G_0}Y,\pi_X^*x\bullet \pi_Y^*y)\in K_{G}^{-p-q}(X\times_{G_0}Y,P_{X\times_{G_0}Y}\otimes Q_{X\times_{G_0}Y})
\end{equation} 
where $\pi_X,\pi_Y$ stand for the respective projections from $X\times_{G_0}Y$ to $X$ and $Y$.

The following is the main result of this paper:
\begin{theorem}\label{mainthm}
For any Lie groupoid $G$ the product on cycles described above gives a well defined bilinear associative product
\begin{equation}\label{Kgeoprodthm}
K_{*}^{geo}(G,\alpha)\times K_{*}^{geo}(G,\beta)\to K_{*}^{geo}(G,\alpha+\beta)
\end{equation}
\end{theorem}

\begin{proof}
We will prove first that the product described above is well defined in the twisted K-homology group. Let $P$ and $Q$ two twistings on $G$. Let $(X,x)$ with $x\in K_{G}^{-p}(X,P_X)$ and $(Y,y)$ with $y\in K_{G}^{-q}(Y,Q_Y)$. Suppose we have smooth maps
$X\stackrel{g}{\longrightarrow}X'$ and $Y\stackrel{f}{\longrightarrow}Y'$. We would finish if we can show that
$$(X\times_{G_0}Y,\pi_X^*x\bullet \pi_Y^*y)\sim (X'\times_{G_0}Y',\pi_{X'}^{*}g!x\bullet \pi_{Y'}^{*}f!y).$$
In fact we can consider the smooth map
$$X\times_{G_0}Y\stackrel{g\times f}{\longrightarrow}X'\times_{G_0}Y'$$ which fits the following commutative diagrams
$$
\xymatrix{
X\times_{G_0}Y\ar[d]_-{\pi_X}\ar[r]^-{g\times f}&
X'\times_{G_0}Y'\ar[d]^-{\pi_{X'}}\\
X\ar[r]_-{g}&X'
}
$$
and 
$$
\xymatrix{
X\times_{G_0}Y\ar[d]_-{\pi_Y}\ar[r]^-{g\times f}&
X'\times_{G_0}Y'\ar[d]^-{\pi_{Y'}}\\
Y\ar[r]_-{f}&Y'
}
$$
The result now follows from proposition \ref{shrieknatural} and proposition \ref{shriekring} since they imply
$$(g\times f)!(\pi_X^*x\cdot \pi_Y^*y)=(g\times f)!(\pi_X^*x)\cdot (g\times f)!(\pi_Y^*y)=\pi_{X'}^*g!x\cdot \pi_{Y'}^*f!y$$
and hence 
$$(X\times_{G_0}Y,\pi_X^*x\cdot \pi_Y^*y)\sim (X'\times_{G_0}Y',\pi_{X'}^{*}g!x\cdot \pi_{Y'}^{*}f!y),$$
and hence the product is well defined.

Now, we prove the associativity, the fact that it is bilinear being immediate. Take three cycles $(X,x), (Y,y), (Z,z)$ with $x,y,z$ in the respective Twisted K-theory groups associated to $PU(H)$-principal bundles $P_X,Q_Y, R_Z$. It is enough to prove that the element 
$$\pi_{X\times_{G_0}Y}^*(\pi_X^*x\bullet \pi_Y^*y)\bullet \pi_Z^*z\in  K_{G}^{-p-q-r}(X\times_{G_0}Y\times_{G_0}Z,P_X\otimes Q_Y\otimes R_Z)$$
coincides with the element
$$\pi_X^*x\bullet \pi_{Y\times_{G_0}Z}^*(\pi_Y^*y\bullet \pi_Z^*z),$$
and this is a direct computation following corollary \ref{prodFredassoc}, lemma \ref{pullnatural} and proposition \ref{pullring} above. This concludes the proof.
\end{proof}

 
\section{Transfering the product via the Baum-Connes map} 
Recall that in \cite{CaWangBC} the Baum-Connes assembly map
\begin{equation}
\xymatrix{
K_{*}^{geo}(G,\alpha)\ar[rr]^-{\mu_\alpha}&&K^{-*}(G,\alpha)
}
\end{equation}
was constructed for every twisting $\alpha$ on $G$ where $K^{*}(G,\alpha):=
K_{-*}(C_r^*(G,\alpha))$ stands for the $K-$theory of the reduced $C^*$-algebra associated to the twisted groupoid $(G,\alpha)$, (there is also the assembly map taking values on the maximal $C^*$-algebra). The definition of the Baum-Connes map is given by 
$$\mu_\alpha(X,x):=\pi_X!(x)\in K^{*}(G,\alpha)$$
where $\pi_X!$ is the pushforward map defined in \cite{CaWangBC}. 

By the results above one could expect to transpose the multiplicative structure via the assembly map. This is of course the case when this twisted Baum-Connes map is an isomorphism. We can write a precise statement.

\begin{corollary}\label{CorBCextprod}
If $G$ is a Lie groupoid for which the geometric Baum-Connes assembly map $\mu_\alpha$ is an isomorphism for every $\alpha\in H^1(G;PU(H))$, then we have a unique bilinear associative structure on the Twisted $K$-theory groups
\begin{equation}\label{prodBC}
K^{*}(G,\alpha)\times K^{*}(G,\beta)\to K^*(G,\alpha+\beta)
\end{equation}
compatible with the structure of \ref{mainthm} via the assembly maps.
\end{corollary}

It would be enough to have that the assembly maps are injective to ensure a multiplicative structure on the images.

Now, the corollary above seems to ask too much but it can be significantly simplified since by corollary 7.2 in \cite{CaWangBC} the morphism $\mu_\alpha$ is an isomorphism if and only if the assembly map for the associated extension groupoid is. In particular if the geometric assembly map coincides with the analytic assembly map, one might expect that for groupoids (or groups) for which the analytic assembly is known to be an isomorphism for the respective extensions we do have that the multiplicative structure above transfer to the $K$-theory counterpart. This is for example the case for (Hausdorff) Lie groupoids satisfying the Haagerup property (\cite{Tu2} theorem 9.3, see also \cite{Tu3} theorem 6.1). We have then to understand the comparison between the geometric and analytic assemblies which are expected to coincide whenever they do in the untwisted case (for discrete groups and for Lie groups they coincide \cite{BHS} and \cite{BOOSW}). We will study the comparison maps between different $K$-homology theories in a further work.

 \section{The Total Twisted K-groups}

The external products treated up to now suggest a ring structure reflecting in twistings the group structure of $H^1(G; PU(H))$. Let us discuss this higher structure in this last section.

As already mentioned above, the group $K_{*}^{geo}(G,\alpha)$ (as its $K$-theoretical counterpart) is well defined up to isomorphism. Indeed, for defining it we have make a choice of a $PU(H)$-principal bundle $P_\alpha$ over the units of $G$ and $G-$equivariant whose isomorphism class is $\alpha$, we have then pullbacked this bundle to every $G$-manifold. Now, if $P_\alpha$ and $P'_\alpha$ are two isomorphic $G-$bundles in the class $\alpha$ there is a group isomorphism
\begin{equation}\label{choiceiso}
K_{*}^{geo}(G,P_\alpha)\cong K_{*}^{geo}(G,P'_\alpha)
\end{equation}
where we have added the notation $P_\alpha$ in the group to emphasize its dependance on the principal bundle. The following statement follows directly form proposition \ref{propprodwelldef}

\begin{proposition}
For any Lie groupoid $G$ the product 
\begin{equation}
K_{*}^{geo}(G,\alpha)\times K_{*}^{geo}(G,\beta)\to K_{*}^{geo}(G,\alpha+\beta)
\end{equation}
described in theorem \ref{mainthm} is compatible with the isomorphisms (\ref{choiceiso}) above.
\end{proposition}

Consider the Total twisted geometric K-homology group of a Lie groupoid $G$, defined as
\begin{equation}
K_{TW,*}^{geo}(G):=\bigoplus_{\alpha\in H^1(G,PU(H))} K_{*}^{geo}(G,\alpha)
\end{equation}

The groups $K_{TW,*}^{geo}(G)$ are well defined up to isomorphism, there is no canonical choice for a representative in a given isomorphism class. These groups and their associated multiplicative structures appeared first in \cite{AdemRuan} in the setting of Orbifolds (definition 8.1 loc.cit). Last proposition allow us however to give a sense to the product, in other terms we can summarize theorem \ref{mainthm} as follows.

\begin{corollary}
For any Lie groupoid, the product described above induces 
\begin{itemize}
\item a ring structure on the even Total twisted geometric K-homology group
$K_{TW,0}^{geo}(G),$ and\\

\item a $K_{TW,0}^{geo}(G)$-module structure on the odd Total twisted geometric K-homology group
$K_{TW,1}^{geo}(G)$.
\end{itemize}
\end{corollary}

We can also consider the Total Twisted K-theory group
$$K^{*}_{TW}(G):=\bigoplus_{\alpha\in H^1(G,PU(H))}K^{*}(G,\alpha).$$
By the theorem above we have a ring (module for the odd case) structure on the image
of the Total twisted Baum-Connes assembly map
\begin{equation}
\xymatrix{
K_{TW,0}^{geo}(G)\ar[rr]^-{\mu_{TW}}&&K^{0}_{TW}(G)
}
\end{equation}
where $\mu_{TW}:=\oplus \mu_\alpha$ whenever $\mu_{TW}$ is injective. In particular if $\mu_{TW}$ is an isomorphism then $K^{0}_{TW}(G)$ has a ring (module for the odd case) structure such that $\mu_{TW}$ is a ring isomorphism.

\bibliographystyle{abbrv}
\bibliography{atiyahsegal}

\begin{thebibliography}{10}

\bibitem{AdemRuan}
A.~Adem and Y.~Ruan.
\newblock Twisted orbifold {$K$}-theory.
\newblock {\em Comm. Math. Phys.}, 237(3):533--556, 2003.

\bibitem{AdemRuanZhang}
A.~Adem, Y.~Ruan, and B.~Zhang.
\newblock A stringy product on twisted orbifold {$K$}-theory.
\newblock {\em Morfismos}, 11(2):33--64, 2007.

\bibitem{atiyahsegal}
M.~Atiyah and G.~Segal.
\newblock Twisted {$K$}-theory.
\newblock {\em Ukr. Mat. Visn.}, 1(3):287--330, 2004.

\bibitem{barcenasespinozajoachimuribe}
N.~B{\'a}rcenas, J.~Espinoza, M.~Joachim, and B.~Uribe.
\newblock Universal twist in equivariant {$K$}-theory for proper and discrete
  actions.
\newblock {\em Proc. Lond. Math. Soc. (3)}, 108(5):1313--1350, 2014.

\bibitem{BHS}
P.~Baum, N.~Higson, and T.~Schick.
\newblock A geometric description of equivariant {$K$}-homology for proper
  actions.
\newblock In {\em Quanta of maths}, volume~11 of {\em Clay Math. Proc.}, pages
  1--22. Amer. Math. Soc., Providence, RI, 2010.

\bibitem{BOOSW}
P.~Baum, H.~Oyono-Oyono, T.~Schick, and M.~Walter.
\newblock Equivariant geometric {$K$}-homology for compact {L}ie group actions.
\newblock {\em Abh. Math. Semin. Univ. Hambg.}, 80(2):149--173, 2010.

\bibitem{CWfusion}
A.~L. Carey and B.-L. Wang.
\newblock Fusion of symmetric {D}-branes and {V}erlinde rings.
\newblock {\em Comm. Math. Phys.}, 277(3):577--625, 2008.

\bibitem{Ca4}
P.~Carrillo~Rouse.
\newblock A {S}chwartz type algebra for the tangent groupoid.
\newblock In {\em {$K$}-theory and noncommutative geometry}, EMS Ser. Congr.
  Rep., pages 181--199. Eur. Math. Soc., Z\"urich, 2008.

\bibitem{CaWangBC}
P.~Carrillo~Rouse and B.~Wang.
\newblock Geometric baum-connes assembly map for twisted differentiable stacks.
\newblock {\em Ann. Sci. \'Ecole Norm. Sup. (4)}, 49:305--351, 2016.

\bibitem{CaWangAdv}
P.~Carrillo~Rouse and B.-L. Wang.
\newblock Twisted longitudinal index theorem for foliations and wrong way
  functoriality.
\newblock {\em Adv. Math.}, 226(6):4933--4986, 2011.

\bibitem{HF}
M.~Del~Hoyo and R.~Loja~Fernandes.
\newblock Riemannian metrics on lie groupoids.
\newblock {\em Arxiv preprint 1404.5989}.

\bibitem{FHT1}
D.~S. Freed, M.~J. Hopkins, and C.~Teleman.
\newblock Loop groups and twisted {$K$}-theory {I}.
\newblock {\em J. Topol.}, 4(4):737--798, 2011.

\bibitem{higson1987}
N.~Higson.
\newblock A characterization of {$KK$}-theory.
\newblock {\em Pacific J. Math.}, 126(2):253--276, 1987.

\bibitem{HS}
M.~Hilsum and G.~Skandalis.
\newblock Morphismes {$K$}-orient\'es d'espaces de feuilles et fonctorialit\'e
  en th\'eorie de {K}asparov (d'apr\`es une conjecture d'{A}. {C}onnes).
\newblock {\em Ann. Sci. \'Ecole Norm. Sup. (4)}, 20(3):325--390, 1987.

\bibitem{lawsonmichelson}
H.~B. Lawson, Jr. and M.-L. Michelsohn.
\newblock {\em Spin geometry}, volume~38 of {\em Princeton Mathematical
  Series}.
\newblock Princeton University Press, Princeton, NJ, 1989.

\bibitem{LeGall}
P.-Y. Le~Gall.
\newblock Th\'eorie de {K}asparov \'equivariante et groupo\"\i des. {I}.
\newblock {\em $K$-Theory}, 16(4):361--390, 1999.

\bibitem{MM}
I.~Moerdijk and J.~Mr{\v{c}}un.
\newblock {\em Introduction to foliations and {L}ie groupoids}, volume~91 of
  {\em Cambridge Studies in Advanced Mathematics}.
\newblock Cambridge University Press, Cambridge, 2003.

\bibitem{Elkaioum}
E.-k.~M. Moutuou.
\newblock Equivariant {$KK$}-theory for generalised actions and {T}hom
  isomorphism in groupoid twisted {$K$}-theory.
\newblock {\em J. K-Theory}, 13(1):83--113, 2014.

\bibitem{Mr}
J.~Mr{\v{c}}un.
\newblock Functoriality of the bimodule associated to a {H}ilsum-{S}kandalis
  map.
\newblock {\em $K$-Theory}, 18(3):235--253, 1999.

\bibitem{PPT}
M.~J. Pflaum, H.~Posthuma, and X.~Tang.
\newblock Geometry of orbit spaces of proper {L}ie groupoids.
\newblock {\em J. Reine Angew. Math.}, 694:49--84, 2014.

\bibitem{Tu2}
J.-L. Tu.
\newblock La conjecture de {B}aum-{C}onnes pour les feuilletages moyennables.
\newblock {\em $K$-Theory}, 17(3):215--264, 1999.

\bibitem{Tu3}
J.-L. Tu.
\newblock The {B}aum-{C}onnes conjecture for groupoids.
\newblock In {\em $C\sp *$-algebras (M\"unster, 1999)}, pages 227--242.
  Springer, Berlin, 2000.

\bibitem{TXring}
J.-L. Tu and P.~Xu.
\newblock The ring structure for equivariant twisted {$K$}-theory.
\newblock {\em J. Reine Angew. Math.}, 635:97--148, 2009.

\bibitem{tuxustacks}
J.-L. Tu, P.~Xu, and C.~Laurent-Gengoux.
\newblock Twisted {$K$}-theory of differentiable stacks.
\newblock {\em Ann. Sci. \'Ecole Norm. Sup. (4)}, 37(6):841--910, 2004.

\bibitem{zung}
N.~T. Zung.
\newblock Proper groupoids and momentum maps: linearization, affinity, and
  convexity.
\newblock {\em Ann. Sci. \'Ecole Norm. Sup. (4)}, 39(5):841--869, 2006.

\end{thebibliography}

\end{document}